\documentclass[reqno,11pt]{amsart}
\usepackage[utf8]{inputenc}
\usepackage{graphicx}
\usepackage{slashed}
\usepackage{amscd}
\usepackage{amssymb}
\usepackage{esint}
\usepackage{enumitem}
\usepackage[T1]{fontenc}
\usepackage{subcaption}

\usepackage[percent]{overpic}

\usepackage{pst-node}
\usepackage{tikz-cd}

\usepackage{comment}
\usepackage{amsmath}
\usepackage{dsfont}
\usepackage[mathscr]{eucal}
\ifx\makepics\undefined

\usepackage{hyperref}
\else
\usepackage[pspdf=-dALLOWPSTRANSPARENCY]{auto-pst-pdf} 
\fi
\usepackage{epsfig}
\usepackage{pstricks}
\usepackage{pst-all}
\def\Inf{\operatornamewithlimits{inf\vphantom{p}}}

\numberwithin{equation}{section}
\allowdisplaybreaks[1]
\definecolor{labelkey}{gray}{.65}

 \author[M. van den Beld-Serrano]{Marco van den Beld-Serrano \\ \\ June 2026}
 \address{Fakult\"at f\"ur Mathematik \\ Universit\"at Regensburg \\ D-93040 Regensburg \\ Germany}
 \email{marco.van-den-beld-serrano@ur.de}

\renewcommand{\div}{{\rm{div}}}
\newtheorem{Def}{Definition}[section]
\newtheorem{Thm}[Def]{Theorem}

\newtheorem{Prp}[Def]{Proposition}
\newtheorem{Lemma}[Def]{Lemma}
\newtheorem{Remark}[Def]{Remark}

\newtheorem{Corollary}[Def]{Corollary}
\newtheorem{Example}[Def]{Example}

\newcommand{\Thanks}{\vspace*{.5em} \noindent \thanks}
\newcommand{\beq}{\begin{equation}}
\newcommand{\eeq}{\end{equation}}
\newcommand{\Proof}{\begin{proof}}

\newcommand{\R}{\mathbb{R}}

\newcommand{\N}{\mathbb{N}}

\DeclareFontFamily{OT1}{rsfso}{}
\DeclareFontShape{OT1}{rsfso}{m}{n}{ <-7> rsfso5 <7-10> rsfso7 <10-> rsfso10}{}
\DeclareMathAlphabet{\myscr}{OT1}{rsfso}{m}{n}

\newcommand{\Fpq}{F^{\alpha}_{p,q}}
\newcommand{\Tspace}{\mathcal{T}^{p}}
\newcommand{\Tt}{\mathcal{T}^{p,t}}

\title{The alignment time function}

\begin{document}

\begin{abstract}
    Given a fixed past-directed timelike vector field, does there exist a time function whose gradient is optimally aligned with it? We address this question by introducing a functional that, on the one hand, captures the misalignment between the timelike vector field and the gradients of suitable Sobolev functions, and, on the other hand, penalizes null gradients. Our analysis focuses on compact subsets of smooth stably causal spacetimes. More precisely, we prove that, under suitable assumptions on the Sobolev index and the strength of the null gradient penalization, there exists a unique smooth temporal function which minimizes the considered functional. We refer to this minimizer as the \emph{alignment time function}. Furthermore, several useful properties of the alignment time function are established: there exists a canonical procedure to improve its steepness, it is stable under $C^{p}$ convergence of the underlying metrics and vector fields and it inherits the symmetries shared by the metric and the given vector field. 
\end{abstract}

\maketitle
\tableofcontents

\section{Introduction}

A well-known result in differential geometry, the Frobenius theorem, implies that not every vector field is necessarily the gradient of a smooth function. However, does there still exist a function whose gradient, in a suitable sense, best aligns with it? And, if yes, is it unique? What properties does such a function present? We explore these questions in the context of compact subsets of Lorentzian manifolds and with the following twofold aim: 
\begin{itemize}[leftmargin=2em]
    \item[(i)] In the first place, we aim to select the function whose gradient is optimally aligned with respect to a fixed past-directed timelike vector field $u$.
    \item[(ii)] Secondly, the minimizer should present, in a suitable sense, an improved steepness with respect to $u$. Intuitively, this means that the gradient of the minimizer should be bounded further away from the lightcone than $u$.
\end{itemize}
Quite remarkably, these guidelines require simultaneously tools from Riemannian and Lorentzian geometry, and the calculus of variations: on the one hand, (i) requires to introduce a Riemannian metric $h$ and a functional depending on the corresponding Sobolev spaces which quantifies the misalignment between a function's gradient and $u$. On the other hand, guideline (ii) can be successfully addressed because of the particular features of the Lorentzian norm (more precisely, the reverse triangle inequality guarantees strict convexity of a null-gradient penalizing functional). Finally, the direct method in the calculus of variations plays a crucial role in proving existence of a minimizer.

From the physical perspective it is important to remark that several cosmological or quantum gravity models assume the existence of a privileged timelike vector field. For instance, the theory of causal fermion systems (see \cite{baryogenesis1,baryogenesis2,baryogenesis3}) the vector-tensor theories of gravity (e.g. the bumblebee model, see \cite{Kostelecky,Bumblebee2025}) or the Einstein-aether theory of gravity (\cite{Jacobson1,Jacobson2}). Therefore, the outcomes of this project entail existence and uniqueness of a time function adapted to such a privileged timelike vector field in these different theories. On the other hand, note that some authors directly assume the existence of a preferred time function or a preferred foliation (see \cite{Horava1}), or study its implications to causality theory (see \cite{Carballo,Bhattacharyya}). 
Hence, the presented construction also bridges the gap between the assumption of a preferred timelike vector field and a preferred time function. It is to be observed that also in the foundational research in General Relativity, the requirement of background independence remains open to debate (\cite{Vassallo}). In any case, note that there exist tight experimental constraints on Lorentz violation (cf. \cite{KosteleckyRusell}). 
\subsection{Outline of the main results}\label{sec:outline}
\hfill

 \vspace{0.2cm}

We present an overview of the setup and main results of this paper. Let $(M,g)$ be a smooth stably causal spacetime, $\Omega\subset M$ a connected compact subset with Lipschitz boundary and $u$ a smooth past-directed timelike vector field. Upon the choice of a Riemannian metric $h$ (which can be constructed solely from $g$ and $u$, cf. \eqref{eq:Riemannian_metric}), consider the set of Sobolev functions $\Tspace\subset H^{p}(\Omega)$ whose gradient $\nabla^{g}t$ is a.e. past-directed causal or vanishes in $\Omega$ and which satisfy the zero mean condition (i.e. $\int_{\Omega}td\mu_{h}=0$ for $t\in \Tspace$). 

The \emph{misalignment functional} $\Fpq:\Tspace\rightarrow[0,\infty]$ is given by 
\begin{align*}
    \Fpq(t):=\begin{cases}
        \|u-\nabla^{g} t\|^{2}_{H^{p-1}(\Omega)} & \textrm{if }\alpha=0\;,\\
        \|u-\nabla^{g} t\|^{2}_{H^{p-1}(\Omega)}+\alpha\int_{\Omega}\frac{1}{|\nabla^{g}t|_{g}^{2q}}d\mu_{h} & \textrm{if } \alpha>0\;,
    \end{cases}\quad
\end{align*}
where $p,q\in\N$ and we set $\Fpq(t)=\infty$ if $\alpha>0$ and $t\in\Tspace$ is not a.e. timelike. This functional addresses simultaneously goals (i) and (ii) sketched above: the first term in $\Fpq$ quantifies the misalignment between $u$ and the gradient of functions in $\Tspace$ whereas the second term penalizes (for $\alpha>0$) functions whose gradient is null or vanishes on positive measure subsets of $\Omega$ and penalizes closeness to the lightcone. However, the choice of a functional and set $\Tspace$ satisfying the desired features is not obvious and, consequently, in Remark \ref{rem:optimal_regularity} and Example \ref{ex:alternative_functionals} the different ingredients appearing in the variational problem are discussed in detail. 

We can now state our main result: if the Sobolev index $p$ and the penalization index $q$ are sufficiently large, then there exists for all $\alpha>0$ a unique temporal function $t_{\alpha}$, smooth in $\Omega^{\circ}$, which minimizes the misalignment functional. We refer to it as the \emph{alignment time function}. Moreover, if $u$ is of gradient form, then the minimizer $t_{0}$ of $F^{0}_{p,q}$ satisfies that $u=\nabla^{g}t_{0}$. The precise statement of this result is the following:
\begin{Thm}\label{theo:main_result_compact}
    Let $(M,g)$ be an $n$-dimensional smooth stably causal spacetime, $u$ a smooth past-directed timelike vector field and $\Omega\subset M$ compact and connected with Lipschitz regular boundary. Then,
\begin{itemize}[leftmargin=2em]
    \item[i)] If $\alpha>0$ and $p\geq2$, the functional $\Fpq$ has a unique minimizer $t_{\alpha}\in\Tspace$ with a.e. timelike gradient in $\Omega$ for any $q\in\N$.
    \item[ii)] If $\alpha>0$, $\gamma\in(0,1)$, $p> \frac{n}{2}+1+\gamma$ and $q\geq n/\gamma$, the gradient of the minimizer $t_{\alpha}$ is everywhere timelike in $\Omega$ and $t_{\alpha}\in C^{\infty}(\Omega^{\circ})$.
    \item[iii)] If $\alpha=0$, there exists a unique minimizer $t_{0}\in\Tspace$ of $F^{0}_{p,q}$ for any $p\in \N$ which, if $u$ is of gradient form, satisfies that $u=\nabla^{g}t_{0}$.
\end{itemize}
\end{Thm}
Section \ref{sec:compact_setting} is devoted to the proof of the above theorem. As a first step we establish in Lemma \ref{lem:props_functional} coercivity, strict convexity and weak lower semicontinuity of the functional (on the subset of functions in $\Tspace$ with a.e. timelike gradient). These features enable us to prove existence and uniqueness of minimizers in Theorem \ref{theo:main_result_compact} which, for $\alpha>0$, have an a.e. timelike gradient. Nevertheless, the remaining problem is to upgrade a.e. timelikeness to everywhere timelikeness. This is where the null-gradient penalizing index $q$ comes into play. Proposition \ref{prop:smooth_minimizer} shows that if $p$ and $q$ are large enough, then the gradient of the minimizer must be everywhere timelike. In particular, under these conditions, $\Fpq$ is Gateaux differentiable and the associated interior Euler-Lagrange equation is uniformly elliptic (Proposition \ref{prop:smooth_minimizer}). Interior smoothness of the minimizer follows. Note that the interior Euler-Lagrange equation has a non-linear lower order term which, up to a constant, corresponds to the $p$-d'Alembertian studied extensively in the last years (e.g. see \cite{Mondino_Suhr,octet}).

It is important to remark that, for $\alpha>0$, the gradient of the minimizer $t_{\alpha}$ need not agree with $u$ (even if it is of gradient form). This is motivated by the mentioned aim (ii). In particular, if $u$ is of gradient form, the gradient of $t_{\alpha}$ presents an improved average steepness with respect to $u$ in the following sense (cf. Remark \ref{rem:bounded_gradient})
\begin{align}\label{eq:average_steepness}
    \int_{\Omega}\frac{1}{|\nabla^{g}t_{\alpha}|_{g}^{2q}}d\mu_{h}\leq\int_{\Omega}\frac{1}{|u|_{g}^{2q}}d\mu_{h}\;,
\end{align}
but this cannot be upgraded to a pointwise comparison between $|u|_{g}$ and $|\nabla^{g}t_{\alpha}|_{g}$.

Given the non-constructive nature of the existence proof for the alignment time function, in Section \ref{sec:properties} the emphasis is on determining further features of this temporal function. These properties can be summarized as follows:
\begin{itemize}[leftmargin=2em]
    \item In the first place, as the parameter $q$ describes the strength of the null gradient penalization, it is of interest to consider the sequence $(t_{p,q})_{q\in\N}$ of minimizers of $\Fpq$ (for fixed $p\geq 2$, $\alpha>0$). Proposition \ref{prop:q_infinity} and Corollary \ref{cor:Hp_convergence} show $H^{p}$-convergence of the sequence and an a.e. steepness estimate for the limiting function $t_{p,\infty}\in \Tspace$:
    \begin{align*}
        \limsup_{q\to\infty}\underset{x\in\Omega}{\text{\rm{ess inf}}}\;\big(|\nabla^{g}t_{p,q}|_{g}(x)\big)\leq\underset{x\in\Omega}{\text{\rm{ess inf}}}\;\big(|\nabla^{g}t_{p,\infty}|_{g}(x)\big)\in[1,\infty)\;.
    \end{align*}
    Note that, in general, the infimum of $|\nabla^{g}t_{p,q}|_{g}$ does not increase monotonically with $q$ (see Example \ref{ex:cylinder}).\\
    For $p\geq 2$, the limiting function $t_{p,\infty}$ is a.e. steep, for $p>n/2$ it is an isotone/causal function (non-decreasing function along future-directed causal curves) and for $p>n/2+1$ it is a temporal function (Remark \ref{rem:tpinfty_isotone}). In addition, $t_{p,\infty}$ is the minimizer of $F^{0}_{p,q}$ on the set of a.e. steep functions in $\Tspace$ with steepness constant $C\geq 1$.\\
    Moreover, if $u$ is of gradient form, then (Corollary \ref{cor:pointwise_bound})
    \begin{align*}
        \underset{x\in\Omega}{\text{\rm{inf}}}\;\big(|u|_{g}(x)\big)\leq \underset{x\in\Omega}{\text{\rm{ess inf}}}\;\big(|\nabla^{g}t_{p,\infty}|_{g}(x)\big)\;,
    \end{align*}
    and an analogous steepness bound holds (Remark \ref{rem:bounded_gradient}) for any past-directed timelike vector field (i.e. even if $u$ is not of gradient form) after multiplying the null gradient penalization functional with a suitable constant (however, for simplicity, we prefer to stick to our original choice of the penalty functional).
    \item Secondly, it is relevant to investigate stability of the minimizer under perturbations in the vector field $u$. For this purpose, we fix a background Riemannian metric $h$ (i.e. not constructed from $u$ and $g$) since then the considered Sobolev norms and $\Tspace$ remain unchanged under perturbations in $u$. With this setup, in Proposition \ref{prop:stability} it is shown that the alignment time function is Lipschitz stable in the $H^{p}(\Omega)$ topology under $H^{p-1}(\Omega)$ perturbations of $u$.
    \item Furthermore, the alignment time function presents an important feature in the context of spacetime convergence. Namely, given $C^{p}$ convergent sequences of Lorentzian metrics $g_{k}$ and vector fields $u_{k}$ (with $p$ sufficiently large) the sequence of corresponding alignment time functions also converges strongly (Proposition \ref{prop:spacetime_convergence}).
    \item Finally, the alignment time function inherits the symmetries of the vector field $u$ and the subset $\Omega$ (Proposition \ref{prop:symmetric}): if there exists an isometry $\Phi: M\to M$ which leaves $u$ and $\Omega$ invariant (i.e. $\Phi(\Omega)=\Omega$ and $d\Phi(u)=u$), then $t_{\alpha}\circ\Phi=t_{\alpha}$. It can be shown that no symmetry group $G\subset \textrm{Isom}(M)$ acting transitively on non-empty open subsets of $\Omega$ (cf. Corollary \ref{cor:symmetry_implications}) can leave $u$ and $\Omega$ invariant. This symmetry preservation property can also be used to give (restrictive) conditions under which $t_{\alpha}$ is a Cauchy temporal function (Corollary \ref{cor:Cauchy}).
    
\end{itemize}

\subsection{Preliminaries}\hfill

 \vspace{0.2cm}

We fix the basic definitions and conventions that will be used throughout this paper. A $n$-dimensional $C^{k}$ spacetime $(M,g)$ is an oriented and time oriented smooth $n$-dimensional manifold $M$, with $n\geq2$, equipped with a $C^{k}$ regular Lorentzian metric $g$. The set of smooth vector fields in $M$ is denoted by $\Gamma(TM)$. The convention $(-,+,\ldots,+)$ is used for the signature of the Lorentzian metric $g$. With this convention, a tangent vector $X_{p}\in T_{p}M\setminus\{0\}$ is spacelike, null or timelike if $g_{p}(X_{p},X_{p})>0,=0,<0$, respectively.

This paper uses several important tools and well-known results from causality theory (see \cite{Minguzzi_Sanchez} or \cite{Minguzzi_causality} for a survey on the topic). In particular, \emph{time functions} (continuous functions which are strictly increasing along future-directed causal curves) and \emph{temporal functions} (smooth or, at least, $C^{1}$ functions with an everywhere past-directed timelike gradient) will play a prominent role. It is important to note that temporal functions are time functions but not vice-versa (e.g. consider the time function $f(t,x^{1},\ldots, x^{n-1})=t^{3}$ on $n$-dimensional Minkowski spacetime $\R^{1,n}$). Additionally, a temporal function $t$ is \emph{steep} if there exists a constant $C>0$ (some authors demand that $C=1$) such that 
\begin{align*}
    -g_{p}(\nabla^{g}t|_p,\nabla^{g}t|_p)\geq C\;,
\end{align*}
for all $p\in M$, where $\nabla^{g}t$ is the gradient vector field of $t$, and a function is \emph{almost everywhere steep} if the inequality holds a.e. in $M$. Since we will only consider temporal functions in compact subsets $\Omega\subset M$, they are automatically steep (in $\Omega$).

Moreover, a spacetime $(M,g)$ is stably causal if there exists another metric $g'$ on $M$ such that $(M,g')$ is causal and $g'$-causal tangent vectors are $g$-timelike (for alternative characterizations of stable causality see the above surveys). For our purposes, the importance of stable causality resides in the fact that it is the lowest rung in the causal ladder ensuring existence of a temporal function (see \cite[Theorem 3.56]{Minguzzi_Sanchez}).

Furthermore, given a Riemannian metric $h$ and a compact subset $\Omega\subset M$, we can define the following Sobolev and $L^{2}$ spaces: for an arbitrary $p\in\N$, $H^{p}(\Omega,h)$ is the space of functions $f:\Omega\subset M\to \R$ whose first $p$-weak derivatives (with respect to the $h$-Levi Civita connection $\nabla^{h}$) satisfy that
\begin{align}
    f\in L^{2}(\Omega,h), \quad |\nabla^{h}f|_{h}\in L^{2}(\Omega,h), \quad\ldots,\quad |(\nabla^{h})^{p}f|_{h}\in L^{2}(\Omega,h)\;,
\end{align}
where the $L^{2}(\Omega,h)$ and norm of $(\nabla^{h})^{i}f$ and $H^{p}(\Omega,h)$ norm of $f$ are defined as
\begin{align*}
    &\|(\nabla^{h})^{i}f\|_{L^{2}(\Omega,h)}^{2}:=\int_{\Omega} |(\nabla^{h})^{i}f|_{h}^{2}d\mu_{h}\;,\quad \|f\|_{H^{p}(\Omega,h)}^{2}:=\sum_{i=0}^{p}\|(\nabla^{h})^{i}f\|_{L^{2}(\Omega,h)}^{2}\;,
\end{align*}
with\footnote{Since $\nabla^{h}f$ denotes the $h$-gradient vector field of $f$, we explicitly write $(\nabla^{h})^{1}f$ to denote the action of $\nabla^{h}$ on the function $f$. Since $|df|_{h}=|\nabla^{h}f|_{h}$, we will often use $\|\nabla^{h} f\|_{L^{2}(\Omega)}$ for $\|(\nabla^{h})^{1} f\|_{L^{2}(\Omega)}$.} $(\nabla^{h})^{0}f:=f$, $(\nabla^{h})^{1}f=df$ and $d\mu_{h}$ the measure induced by the Riemannian metric $h$. See also \cite[Section 2.1]{Hebey} or \cite[Section 2]{Aubin_problems} for the general definition of Sobolev spaces on Riemannian manifolds. One defines analogously Sobolev spaces of general tensor fields. If clear by context, explicit mention to $h$ will be omitted and $H^{p}(\Omega)$ and $L^{2}(\Omega)$ will denote the corresponding spaces.

Finally, also tools from the calculus of variations will be employed. Let $F: X\to \R$ denote a Gateaux differentiable functional, with $X$ a Banach space, and $u\in X$. Then, the Gateaux derivative of $F$ at $u\in X$ in the direction of $v\in X$ will be denoted by 
\begin{align*}
    DF(u)[v]:=\frac{d}{ds}F(u+sv)\Big|_{s=0}\;.
\end{align*}

\section{The main result}\label{sec:compact_setting}
\subsection{The setup}\label{subsec:setup}
\hfill
 \vspace{0.2cm}

Let $(M,g)$ be an $n$-dimensional smooth stably causal spacetime, $u$ a smooth past-directed timelike vector field and $\Omega\subset M$ a connected and compact subset with Lipschitz regular boundary. In particular, the timelike vector field can be used to define the following smooth Riemannian metric (see \cite[p. 39]{HawkingEllis} or \cite[Chapter 5, Lemma 36]{Oneill})
\begin{equation}\label{eq:Riemannian_metric}
    h:=g+2\frac{u^{\flat}\otimes u^{\flat}}{|u|^{2}_{g}}\;.   
\end{equation}
The relation between the geometric properties of $h$ and the pair $(g,u)$ have been analyzed thoroughly, see for example \cite{Aazami,Reddy,Olea}.
Note that, except for Proposition \ref{prop:stability} (where a background Riemannian metric is fixed), only the Riemannian metric \eqref{eq:Riemannian_metric} will be considered.

Given the Riemannian metric $h$ \eqref{eq:Riemannian_metric}, we can consider the corresponding Sobolev spaces $H^{p}(\Omega)$ as discussed in the preliminaries. It is well-known that in compact Riemannian manifolds the Sobolev spaces are independent of the chosen metric (cf. \cite[Proposition 2.3]{Hebey}). However, there are two main reasons why a specific Riemannian metric (and, in particular \eqref{eq:Riemannian_metric}) is fixed. In the first place, note that the functional that will be minimized (and thus its minimizer) will depend on this choice. Hence, constructing the metric $h$ uniquely from the pair $(g,u)$ allows us to tackle our variational problem without introducing a Riemannian metric as additional external input. Secondly, the aim of subsequent work will be to extend the results obtained in this paper to the non-compact setting, where the Sobolev spaces do depend on the metric.

Furthermore, it is to be observed that the concepts introduced in this section, the main existence and uniqueness results (Section \ref{subsec:minimizer}) and the interior Euler-Lagrange equation (Section \ref{sec:elliptic_regularity}) do also hold if one considers the Sobolev spaces $W^{p,r}(\Omega)$ (i.e. the space functions whose first $p$ weak derivatives are in $L^{r}(\Omega)$) with $r\in(1,\infty)$, which are not Hilbert spaces, instead of $H^{p}(\Omega)$. Nevertheless, these results do not hold for $W^{p,\infty}(\Omega)$ as it is not a reflexive Banach space. In this work the space $H^{p}(\Omega)$ is preferred over $W^{p,r}(\Omega)$ (with $r\in(1,\infty)$) since Sections \ref{subsec:stability} and \ref{subsec:limit} do exploit the additional inner product space structure of the former.

In the following definition we introduce the misalignment functional and the sets of functions over which the functional will be minimized.

\begin{Def}\label{def:space_functional}
    Let $\alpha\geq0$ be a real number and $p,q\geq1$ natural numbers. Consider the following sets of functions:
    \begin{align*}
        &\mathcal{T}^{p}:=\{t\in H^{p}(\Omega)|\nabla^{g}t \;\textrm{is past-directed causal or vanishes a.e. in $\Omega$ with}\int_{\Omega}t\;d\mu_{h}=0\},\\
        &\mathcal{T}^{p,t}:=\mathcal{T}^{p}\cap\{|\nabla^{g}t|_{g}\neq0 \;\textrm{a.e. in }\Omega\}\;.
    \end{align*}
    Moreover, the misalignment functional $F^{\alpha}_{p,q} : \mathcal{T}^{p}\subset H^{p}(\Omega)\rightarrow (0,\infty]$ is:
    \begin{align}
    &F^{\alpha}_{p,q}(t)=
    \begin{cases}
        F_{h,p}(t) \quad &\textrm{for } \alpha=0\\
        F_{h,p}(t)+\alpha F_{g,q}(t)&\textrm{for }\alpha>0\;,
    \end{cases}
    \end{align}
where $F_{h,p} :\Tspace\rightarrow[0,\infty)$ and $F_{g,q}:\Tspace\rightarrow(0,\infty]$ are
\begin{align}
    &F_{h,p}(t):=\sum_{i=0}^{p-1}\int_{\Omega}|(\nabla^{h})^{i}(u-\nabla^{g} t)|_{h}^{2}d\mu_{h}=\|u-\nabla^{g} t\|^{2}_{H^{p-1}(\Omega)}\;\label{eq:Fhp},\\
&F_{g,q}(t):=
    \begin{cases}
        \int_{\Omega}\frac{1}{|\nabla^{g}t|_{g}^{2q}}d\mu_{h}\quad &\textrm{if} \;t\in\mathcal{T}^{p,t}\\
        +\infty\quad &\textrm{otherwise}\;.
    \end{cases}\label{eq:functional_def}
\end{align}
\end{Def}
Note that we consider gradients which vanish a.e. in $\Omega$ in the definition of the set $\Tspace$ in order to guarantee closedness of this set. Of course, the presented setup and results also hold if $u$ is future-directed timelike after replacing the gradient $\nabla^{g}t$ with $-\nabla^{g}t$ in the definitions of $\Tspace$ and $\Fpq$ (alternatively, one could switch to the convention $(+,-,\ldots,-)$ for the signature of $g$ and simply replace 'past' by 'future' in the above presented constructions).
Moreover, the condition
\begin{align}\label{eq:zero_mean}
    \int_{\Omega}t\;d\mu_{h}=0
\end{align}
satisfied by functions in $\Tspace$ will be referred to as the \emph{zero-mean condition}. The importance of this condition is discussed in Remark \ref{rem:properties}.

Different remarks are in order. In the first place, since $(M,g)$ is stably causal, there exists a `comparator' temporal function $\tau$ in the set $\Tspace$ for which the value of the functional $\Fpq$ will be finite. Secondly, in the following remark we discuss the assumed smoothness of $g$ and $u$ and the existence of a suitable lower regularity class for which the above functionals are still well-defined.
\begin{Remark}\label{rem:optimal_regularity}\rm{
    Is smoothness of $u$ and $g$ necessary for the above functionals to be well-defined? A priori, for $F_{h,p}$ it suffices that the tensor fields $(\nabla^{h})^{i}(u-\nabla^{g}t)$ are square-integrable\footnote{In this remark, $L^{2}$- or $H^{p}$-regularity is always with respect to some background smooth Riemannian metric, not $h$.} for each $i\leq p-1$ whereas for $F_{g,q}$ it suffices that $|\nabla^{g}t|_{g}^{-2q}$ is measurable (as an extended real valued function). In the first place, this clearly holds if $u$ and $g$ are $C^{p-1}$ regular ($\Omega$ is compact, so the derivatives of a $C^{r}(\Omega)$ function are bounded and its pointwise product with a $H^{r}(\Omega)$ function is again in $H^{r}(\Omega)\subset L^{2}(\Omega)$ by the weak product rule). 
    \par On the other hand, if $u$ and $g$ (so also $h$) are only $H^{p}$-regular, this might not be sufficient: given a $H^{p}$-regular vector field $X$ and metric $h$, $\nabla^{h}X$ involves the product of the $H^{p-1}$-regular $h$-Christoffel symbols and the $H^{p}$-regular coefficient of $X$. And since, in general, Sobolev spaces are not algebras under pointwise multiplication, $\nabla^{h}X$ might not be $H^{p-1}$-regular. A sufficient condition is that $p-1>n/2$. In particular, if $m-1>n/2$ and $r\leq m-1$ the multiplication map
    \begin{align*}
        H^{m-1}(\Omega)\times H^{r}(\Omega)\to H^{r}(\Omega),\quad (f,g)\mapsto f\cdot g\;,
    \end{align*}
    is bounded by the multiplication theorems for Sobolev spaces (e.g. \cite[Theorem 5.1]{BehzadanHolst}; alternatively one can use Moser estimates, see \cite[Proposition 3.7]{TaylorIII}). Hence, for $m\geq p$ and $m>n/2+1$ the following map
    \begin{align*}
        (\nabla^{h})^{i} :H^{p-1}(\Omega)\to H^{p-1-i}(\Omega)\subset L^{2}(\Omega),\quad u-\nabla^{g}t\mapsto (\nabla^{h})^{i}(u-\nabla^{g}t)\;,
    \end{align*}
    is well-defined and bounded for all $i\leq p-1$, which implies well-definedness of $F_{h,p}$ and $F_{g,q}$. 
    
    \par However, if $p\in\{1,2\}$, then the previously mentioned $C^{p-1}$ regularity of $g$ and $u$ is sharper since $H^{m}(\Omega)\subset C^{1}(\Omega)$ for $m>n/2+1$ by the Sobolev embedding theorems. Therefore, a convenient sufficient regularity class for $u$ and $g$ that still ensures well-definedness of the above functionals is
    \begin{align}
        \begin{cases}
            &\textrm{If $p\in\{1,2\}$ let $g,u$ be $C^{p-1}$-regular.}\\
            &\textrm{If $p>2$ let $g,u$ be $H^{m}$-regular with $m\geq p$ and $m>n/2+1$.}
         \end{cases}
    \end{align}
    For simplicity in the rest of the paper we nevertheless assume that $g$ and $u$ are smooth.

}\end{Remark}
Moreover, we discuss the similarities of the null-gradient penalizing functional $F_{g,q}$ to an important functional studied in the literature.
\begin{Remark}\rm{
In \cite[Section 3]{Mccann1} the following functional is introduced:
\begin{align}\label{eq:Mccann_functional}
    L(v):=-\frac{1}{r}|v|_{g}^{r}\;,
\end{align}
if $v$ is future-directed causal and $L(v)=\infty$ otherwise, and $r\in(0,1)$. Then, the Legendre dual of $L(v)$ yields an analogous functional on the cotangent bundle but with the exponent $r'=r/(r-1)<0$. Hence, the integrand of $F_{g,q}$ agrees with the Legendre dual up to a multiplicative constant. In \cite[Lemma 3.1]{Mccann1} it is shown that $L$ is convex. The variational derivative (at a gradient $\nabla^{g}f$) of the corresponding integral functional yields an operator referred to as the \emph{$r$-d'Alembertian}. In \cite{octet} convexity of \eqref{eq:Mccann_functional} is used in order to prove ellipticity of this operator, a feature of central relevance, for example, in \cite{Mccann_Ohanyan}. Taking into account the similarities of the above functional (or its Legendre dual) and $F_{g,q}$, it is no surprise that the non-linear term of the interior Euler-Lagrange equation associated to the misalignment functional (cf. Proposition \ref{prop:smooth_minimizer}) seems, at first sight, very close to the $r'$-d'Alembertian. Actually, as will be discussed in Remark \ref{rem:p-dalembertian}, it agrees, up to a constant, with this operator because of our choice of the Riemannian metric \eqref{eq:Riemannian_metric}.

}\end{Remark}

Furthermore, it is worth discussing why the above functionals and sets of functions are introduced in this particular form.

\begin{Remark}\label{rem:properties}\em{Let us motivate and discuss the specific choices made in the definition of the set of functions $\Tspace$ and the functional $\Fpq$:
    \begin{itemize}[leftmargin=2em]
        \item[i)]\underline{On the set $\mathcal{T}^{p}$:}
        \\ The functional $\Fpq$ is minimized over $\Tspace$ instead of $\Tt$ because only the former is norm-closed, a property necessary in order to apply compactness results in Sobolev spaces. Moreover, $\Tspace$ is not a vector space ($-t\not\in\Tspace$), but it is a convex set and thus norm-closedness implies weak-closedness (\cite[Proposition 1.21]{Peypouquet}).
        \item[ii)]\underline{On the functional $F_{h,p}$:}
        \\The functional $F_{h,p}$ describes the misalignment between $\nabla^{g}t$ and the vector field $u$. In particular, since $F_{h,p}(t)=0$ implies that $\|u-\nabla^{g}t\|_{L^{2}}=0$, for $t\in \Tspace$ the functional $F_{h,p}(t)$ vanishes if and only if $u=\nabla^{g}t$ a.e. in $\Omega$. In particular, such a function $t\in\Tspace$ is unique by the zero-mean condition.
        \\ Of course, if $u=\nabla^{g}f$ but the function $f$ does not satisfy the zero-mean condition, there exists an additive reparametrization $t\in\Tspace$ such that $u=\nabla^{g}t$.
        \item[iii)]\underline{On the functional $F_{g,q}$ and the parameter $\alpha$:}
        \\Due to the possibility that the gradient of a minimizer $t\in \Tspace$ of $F_{h,p}$ is null or even vanishes, the functional $F_{g,q}$ penalizes such outcomes. In particular, existence of a temporal function $\tau$ on $(M,g)$ implies that the gradient of a minimizer $t_{\alpha}\in\Tspace$ of $\Fpq$ (with $\alpha>0$) has to be timelike almost everywhere in $\Omega$ since
 \begin{align*}
     F_{g,q}(t_{\alpha})\leq \frac{1}{\alpha}\Fpq(\tau)<\infty\;.
 \end{align*}
        \par In the specific case that $u$ is of gradient form, the minimizer $t_{\alpha}$ of $\Fpq$ will in general not satisfy that $u=\nabla^{g}t_{\alpha}$ for $\alpha>0$. However, the gradient of $t_{\alpha}$ presents an improved average steepness with respect to $u$ (in the sense of \eqref{eq:average_steepness}), which can be upgraded to a global bound in the case $q\to\infty$ (Corollary \ref{cor:pointwise_bound}). 
        \item[iv)]\underline{On the parameters $p$ and $q$:}
        \\ The parameter $p$ plays an important role in order to control the gradient of the limiting function $t\in\Tspace$ of a bounded sequence of functions in $\Tspace$: choosing $p$ large enough, one can use compactness results to extract a convergent subsequence of (sufficiently regular) gradients from a $H^{p}$-bounded sequence of functions.
        \par On the other hand, even if a minimizer exists and is smooth, its gradient may only be a.e. timelike. Choosing $q$ and $p$ large enough yields an upgrade to everywhere timelikeness (and even an a.e. steepness estimate in the $q\to\infty$ limit).        
        \item[v)]\underline{On the zero-mean condition in $\Tspace$:}
        \\ A condition on the functions in $\Tspace$ is crucial in order to guarantee coercivity of $F^{\alpha}_{p,q}$ and uniqueness of minimizers as the functional only depends on weak derivatives of $t\in\Tspace$ (i.e. $\Fpq(t)=\Fpq(t+c)$ for $c\in \R$). A potential alternative condition is
        \begin{align}\label{eq:trace_condition}
            & t|_{\partial\Omega}=\tau|_{\partial\Omega}\;,
        \end{align}
        where the left hand side is well defined by the trace theorem. However, if $u$ is of gradient form, then a function $f$ satisfying $u=\nabla^{g}f$ will in general not fulfill this trace condition (nor can be reparametrized to satisfy it). This motivates our preference for the zero-mean condition \eqref{eq:zero_mean}. 
    \end{itemize}
}\end{Remark}
In addition to the previously mentioned feature of $F_{g,q}$ of allowing to improve a.e. to everywhere timelikeness of the gradient of the minimizer (if $q$ and $p$ are large enough, see Proposition \ref{lem:everywhere_timelike}), there exist further reasons why the choice of a penalty functional is rather subtle. We delve into some of the issues potential alternative penalty functionals present in the rest of this subsection. The following lemma shows the difficulty to construct a functional which satisfies some of the necessary properties in order to use the direct method in the calculus of variations and that vanishes if $u$ is of gradient form. 

\begin{Lemma}
    Let $C^{-}_{x}\subset T_{x}M$ denote the set of past-directed timelike vectors in $T_{x}M$ and $u_{x}\in C^{-}_{x}$. Then, there does not exist a function $F: C^{-}_{x}\to\R$ which satisfies:
    \begin{enumerate}[leftmargin=2em]
        \item[(i)] $F$ is convex
        \item[(ii)] $F$ is non-negative
        \item[(iii)] $F(v)=0$ if and only if $|v|_{g}=|u_{x}|_{g}$.
    \end{enumerate}
\end{Lemma}
\begin{proof}
    Consider the following subset of $T_{x}M$:
    \begin{align}
        V:=\{v\in C^{-}_{x}:|v|_{g}=|u_{x}|_{g}\}\;.
    \end{align}
    We now show that, if conditions $(i)$ and $(ii)$ are satisfied, then $V$ is a proper subset of the zero set of $F$ on $C^{-}_{x}$, contradicting condition $(iii)$. 
    
    The set $V$ contains at least two non-collinear vectors: given an generalized orthonormal basis $\{e_{i}\}_{i=0}^{n}$ of $T_{x}M$, the vectors $v_{1}:=|u_{x}|_{g}e_{0}$ and $v_{2}:=|u_{x}|_{g}(\cosh{(\varphi)}e_{0}+\sinh{(\varphi)}e_{1})$ are non collinear for $\varphi\neq0$ and belong to $V$. The (strict) reversed triangle inequality applied to two non-collinear $v_{1},v_{2}\in V$ implies that $V$ is not a convex set:
    \begin{align*}
        |\frac{1}{2}v_{1}+\frac{1}{2}v_{2}|_{g}>\frac{1}{2}|v_{1}|_{g}+\frac{1}{2}|v_{2}|_{g}=|u_{x}|_{g} \implies \frac{v_{1}+v_{2}}{2}\not\in V\;.
    \end{align*}
    However, the zero set of $F$ is convex, as follows from conditions $(i)$ and $(ii)$. Let $v_{1}$ and $v_{2}$ belong to the zero set of $F$ and $\lambda\in(0,1)$:
    \begin{align*}
        F(\lambda v_{1}+(1-\lambda)v_{2})\leq \lambda F(v_{1})+(1-\lambda)F({v_{2}})=0\implies F(\lambda v_{1}+(1-\lambda)v_{2})=0\;.
    \end{align*}
    Since $V$ is a non-convex subset of the zero set of $F$, which is convex, it is a proper subset.
\end{proof}
Note that conditions $(i)$ and $(ii)$ play a crucial role in the direct method in calculus of variations (and are satisfied by $F_{g,q}$): convexity plays an important in showing weak lower semicontinuity of the functional (and strict convexity guarantees uniqueness of minimizers), whereas non-negativity ensures that the functional is bounded from below. Nevertheless, the previous lemma doesn't fully rule out existence of a functional satisfying these properties and which vanishes if $u$ is of gradient form. In the following remark the penalty functional $F_{g,q}$ is compared with other potential candidates, showing that most of them present important drawbacks. These examples exploit strict convexity of $|\cdot|_{g}^{-2q}$ on $\Tt$ (see Lemma \ref{lem:props_functional}).
\begin{Example}\label{ex:alternative_functionals}\rm{
Some examples of null-gradient penalizing functionals that fail to satisfy one of the two first conditions of the above lemma are 
\begin{align*}
    &t\in \Tt\mapsto \int_{\Omega}\Big(\frac{1}{|\nabla^{g}t|_{g}^{2q}}-\frac{1}{|u|^{2q}_{g}}\Big)d\mu_{h},\quad t\in\Tspace\setminus\Tt\mapsto+\infty\;,\\
    &t\in \Tt\mapsto \int_{\Omega}\Big(\frac{1}{|\nabla^{g}t|_{g}^{2q}}-\frac{1}{|u|^{2q}_{g}}\Big)^{2}d\mu_{h},\quad t\in\Tspace\setminus\Tt\mapsto+\infty\;,
\end{align*}
where the first functional can be negative and the second one is not convex. 

Alternatively, one could try to find functionals which satisfy conditions $(i)$ and $(ii)$ of the above lemma and that vanish if $u=\nabla^{g}t$. For example,
\begin{align*}
    t\in\Tspace\mapsto\begin{cases}
        0 &\textrm{if} \;t\in\mathcal{T}^{p,t}\\
        +\infty\quad &\textrm{otherwise}\;.
    \end{cases}
\end{align*}
It is clear that a minimizer of a functional which includes such a penalty term would have a timelike a.e. gradient. However, this functional presents an important disadvantage: since it does not provide a quantitative penalty for approaching the lightcone, almost everywhere timelikeness of the gradient of the minimizer cannot be upgraded to everywhere timelikeness. 

Finally, another functional similar to $F_{g,q}$ is the one given by the Bregman divergence or distance of $|\cdot|_{g}^{-2q}$ (see \cite[Section 1.1.3]{ButnariuIusem}). In particular, for $t\in \Tt$ define
\begin{align*}
    F^{B}_{g,q}(t)&:=
        \int_{\Omega}\Big(\frac{1}{|\nabla^{g}t|_{g}^{2q}}-\frac{1}{|u|_{g}^{2q}}-D(|\cdot|^{-2q}_{g})(u)[\nabla^{g}t-u]\Big)d\mu_{h}\\
        &=\int_{\Omega}\Big(\frac{1}{|\nabla^{g}t|_{g}^{2q}}-\frac{1}{|u|_{g}^{2q}}-2q\frac{g(u,\nabla^{g}t-u)}{|u|_{g}^{2q+2}}\Big)d\mu_{h}\;,
\end{align*}
and $F^{B}_{g,q}(t)=\infty$ if $t\in\Tspace\setminus\Tt$. Strict convexity of $|\cdot|_{g}^{-2q}$ guarantees that this functional is non-negative and convex, and it vanishes if and only if $u=\nabla^{g}t$ a.e.. Moreover, the arguments used in Proposition \ref{lem:everywhere_timelike} in order to upgrade almost everywhere to everywhere timelikeness also apply to this functional. However, in this paper the functional $F_{g,q}$ is preferred to $F^{B}_{g,q}$ since using the former, even when $u$ is of gradient form, the gradient of the minimizer $t_{\alpha}$ presents an improved average steepness with respect to $u$.
}\end{Example}

\subsection{Existence and uniqueness of minimizers}\label{subsec:minimizer}
\hfill
\vspace{0.2cm}

\noindent The main results of this subsection (Theorem \ref{theo:main_result_compact} and Corollary \ref{coro:unique_minimizer}) will be proving the existence and uniqueness of a, at least $C^{1,\gamma}$-regular, temporal function which minimizes the misalignment functional. This temporal function, whose gradient is optimally aligned with $u$, will be called the \emph{alignment time function}.

In the following lemma we show that $\Fpq$ satisfies the necessary features in order to apply the direct method in the calculus of variations. Quite remarkably, the null penalty functional $F_{g,q}$ is strictly convex on $\Tt$ because of the Lorentzian signature of $g$ (in particular, due to the reverse triangle inequality).

\begin{Lemma}\label{lem:props_functional}
    Let $p,q\geq1$ be arbitrary. The functionals $F_{h,p}$ and $F_{g,q}$ satisfy the following properties:
    \begin{itemize}[leftmargin=2em]
        \item[i)] $F_{h,p}$ is a strictly convex, continuous and coercive functional.
        \item[ii)] $F_{g,q}$ is convex on $\Tspace$ and strictly convex on $\Tt$.
        \item[iii)] If $p\geq 2$, then $F_{g,q}$ is weakly lower semicontinuous on $\Tt$.
    \end{itemize}
    Hence, for $\alpha=0$, the functional $\Fpq$ is coercive, strictly convex and weakly lower semicontinuous. If $\alpha>0$ and $p\geq2$, then $\Fpq$ enjoys the same properties on $\Tt$.
    
\end{Lemma}
\begin{proof}
In the first place, coercivity of the functional $\Fpq$ follows from coercivity of $F_{h,p}$. Applying the Poincaré-Wirtinger inequality (cf. \cite[Lemma 3.8]{Hebey}) to the functions in $\Tspace$ (so they satisfy the zero-mean condition) gives
\begin{align*}
    \|t\|_{H^{p}(\Omega)}^{2}&=\|t\|_{L^{2}(\Omega)}^{2}+\|\nabla^{h} t\|_{H^{p-1}(\Omega)}^{2}\leq C\|\nabla^{h} t\|_{L^{2}(\Omega)}^{2}+\|\nabla^{h}t\|_{H^{p-1}(\Omega)}^{2}\leq C\|\nabla^{h}t\|_{H^{p-1}(\Omega)}^{2}
\end{align*}
with $C>0$ a constant depending on $\Omega$, $p$ and $h$. Note that, in this proof, constants might change from one inequality to another but, as is customary, will still be denoted with the same letter. Since $\Omega$ is compact and $u$ and $g$ are smooth
\begin{align*}
    \Fpq(t)\geq F_{h,p}(t)\geq C\|\nabla^{g} t\|_{H^{p-1}(\Omega)}^{2}-\|u\|_{H^{p-1}(\Omega)}^{2}\geq C\|\nabla^{h} t\|_{H^{p-1}(\Omega)}^{2}-C'\geq C\|t\|_{H^{p}(\Omega)}^{2}-C'
\end{align*}
where we used that by compactness of $\Omega$ and smoothness of $g$ and $h$ there exists a constant $c$ such that $\|\nabla^{g} t\|_{H^{p-1}(\Omega)}^{2}\geq c\|\nabla^{h} t\|_{H^{p-1}(\Omega)}^{2}$. Finally, $C,C'>0$ depend also on $u$ and $g$. Hence, $\Fpq$ and $F_{h,p}$ are coercive.

Secondly, let us address convexity of $F_{h,p}$ and $F_{g,q}$. Note that for any $t_{1},t_{2}\in \Tspace$, $t_{1}\neq t_{2}$ implies that $\nabla^{g}t_{1}\neq\nabla^{g}t_{2}$ (if $\nabla^{g}t_{1}=\nabla^{g}t_{2}$ a.e., then $t_{1}=t_{2}+C$ a.e. but $C=0$ by the zero mean condition). Strict convexity of the $H^{p-1}(\Omega)$ norm squared entails strict convexity of $F_{h,p}$.

For the proof of convexity of $F_{g,q}$, consider arbitrary $t_{1},t_{2}\in\Tspace$ with $t_{1}\neq t_{2}$. Convexity of $F_{g,q}$ is immediate if the gradient of one of them is null or vanishes on a positive measure subset of $\Omega$. Assume now that both $\nabla^{g}t_{1}$ and $\nabla^{g}t_{2}$ are past-directed timelike a.e.. The reversed triangle inequality gives for any $\lambda\in(0,1)$
\begin{align}\label{eq:reversed_triangle}
    |\lambda\nabla^{g}t_{1}+(1-\lambda)\nabla^{g}t_{2}|_{g}\geq \lambda|\nabla^{g}t_{1}|_{g}+(1-\lambda)|\nabla^{g}t_{2}|_{g}\;,
\end{align}
which yields strict convexity on $\Tt$. First consider the case with $|\nabla^{g}t_{1}|_{g}\neq|\nabla^{g}t_{2}|_{g}$,
\begin{align*}
    F_{g,q}(\lambda t_{1}+(1-\lambda)t_{2})
    &\leq \int_{\Omega}\frac{1}{(\lambda|\nabla^{g}t_{1}|_{g}+(1-\lambda)|\nabla^{g}t_{2}|_{g})^{2q}}d\mu_{h}\\
    &< \lambda\int_{\Omega}\frac{1}{|\nabla^{g}t_{1}|_{g}^{2q}}d\mu_{h} +(1-\lambda)\int\frac{1}{|\nabla^{g}t_{2}|_{g}^{2q}}d\mu_{h}\\
    &=\lambda F_{g,q}(t_{1})+(1-\lambda)F_{g,q}(t_{2})\;,
\end{align*}
where for the second inequality we used that the function $f: (0,\infty)\to(0,\infty), x\mapsto \frac{1}{x^{2q}}$ is strictly convex for any $q\geq1$, so the same holds for $F_{g,q}$ on $\Tt$. On the other hand, if $|\nabla^{g}t_{1}|_{g}=|\nabla^{g}t_{2}|_{g}$ and $\nabla^{g}t_{1}\neq\nabla^{g}t_{2}$ on a positive measure subset, then the two gradients are not collinear on this set and one has a strict inequality in \eqref{eq:reversed_triangle} which gives again strict convexity of $F_{g,q}$.

It remains to prove weak lower semicontinuity of the functionals. Let $i\in\{0,1,\ldots,p-1\}$. The linear map $t\in H^{p}(\Omega)\mapsto (\nabla^{h})^{i}\nabla^{g}t\in L^{2}(\Omega)$ is bounded and thus continuous,
\begin{align*}
    &\|(\nabla^{h})^{i}\nabla^{g}t\|^{2}_{L^{2}(\Omega)}\leq C(\|t\|_{L^{2}(\Omega)}^{2}+\sum_{m=0}^{p-1}\|(\nabla^{h})^{m}\nabla^{h}t\|_{L^{2}(\Omega)}^{2})=C\|t\|_{H^{p}(\Omega)}^{2}\;,
\end{align*}
so also the affine map $\psi_{i} : H^{p}(\Omega)\rightarrow L^{2}(\Omega), t\mapsto (\nabla^{h})^{i}(u-\nabla^{g}t)$ is continuous. Since $F_{h,p}(t)=\sum_{i=0}^{p-1}\|\psi_{i}(t)\|_{L^{2}(\Omega)}^{2}$, also $F_{h,p}$ is continuous and, since it is convex, sequential weak lower semicontinuity follows (see \cite[Proposition 2.17]{Peypouquet}).

Note that, even though the real valued function $t\mapsto|\nabla^{g}t|_{g}^{-2q}$ is continuous on $\Tt$, this does not imply continuity of $F_{g,q}$ on $\Tt$ and hence the previous argument cannot be used here in order to conclude that the functional is weakly lower semicontinuous. For this purpose, the additional assumption that $p\geq 2$ is exploited.  

Consider a sequence $(t_{n})_{n\in\N}$ in $\Tt$ converging weakly in $H^{p}(\Omega)$ to $t_{\ast}\in \Tt$ and a subsequence $(t_{n_{k}})_{k\in\N}$ for which $(F_{g,q}(t_{n_{k}}))_{k\in\N}$ converges to $\liminf_{n\to \infty}F_{g,q}(t_{n})<\infty$ (if it diverges weak lower semicontinuity of $F_{g,q}$ is immediate). 

By the Rellich-Kondrachov theorem $(p\geq2)$ the Sobolev space $H^{p}(\Omega)$ is compactly embedded in $H^{1}(\Omega)$, so (after passing to a further subsequence) $(t_{n_{k}})_{k\in\N}$ converges strongly in $H^{1}(\Omega)$ to $t'_{\ast}\in H^{1}(\Omega)$. From uniqueness of the weak limit (strong $H^{1}(\Omega)$-convergence implies weak $H^{1}(\Omega)$-convergence) it follows that $t_{\ast}'=t_{\ast}$. In particular, $(\nabla^{g}t_{n_{k}})_{k\in\N}$ converges in the $L^{2}(\Omega)$-norm to $\nabla^{g}t_{\ast}$, which (passing again to a subsequence) implies a.e. convergence and so also $(|\nabla^{g}t_{n_{k}}|^{-2q}_{g})_{k\in\N}$ converges to $|\nabla^{g}t_{\ast}|^{-2q}_{g}$ a.e. in $\Omega$ (composition with continuous maps preserves a.e. convergence). Since these are non-negative functions, Fatou's lemma implies weak lower semicontinuity of $F_{g,q}$:
 \begin{align*}
        \liminf_{n\to\infty} F_{g,q}(t_{n})=\liminf_{k\to\infty} F_{g,q}(t_{n_{k}})&\geq\int_{\Omega}\liminf_{k\to\infty}\frac{1}{|\nabla^{g}t_{n_{k}}|_{g}^{2q}}d\mu_{h}=F_{g,q}(t_{\ast})\;.
    \end{align*}
This concludes the proof.

\end{proof}

Note that in the previous lemma it is implicitly used that weak continuity and weak sequential continuity of a convex function on a Banach space are equivalent properties (cf. \cite[Proposition 2.7]{Peypouquet}).

The following theorem is one of the main results of this paper. It proves existence and uniqueness of minimizers of $\Fpq$ for any $\alpha\geq0, q\in\N$ and $p\geq2$.

\begin{Thm}\label{theo:main_result_compact}
    Let $\alpha\geq0$ and $q\in\N$. Then, the following holds
    \begin{itemize}[leftmargin=2em]
        \item[i)] If $\alpha=0$, there exists a unique minimizer $t_{0}\in\Tspace$ of $F^{0}_{p,q}$ for any $p\in\N$.
        \item[ii)] If $\alpha>0$ and $p\geq 1$, there exists a unique minimizer $t_{\alpha}\in \Tt$ of $\Fpq$.
    \end{itemize}
\end{Thm}
\begin{proof}
    In the first place, consider the case with $\alpha>0$. Fix an arbitrary $q\in \N$. Existence of a minimizer follows from the direct method of the calculus of variations and the assumption that $p\geq 2$. 
    
    Consider a sequence $(t_{n})_{n\in\N}$ in $\Tspace$ such that $(\Fpq(t_{n}))_{n\in\N}$ converges to 
    \begin{align}
        F_{\inf}:=\inf_{t\in \Tspace} \Fpq(t)<\infty\;.
    \end{align} 
    Since $(\Fpq(t_{n}))_{n\in\N}$ is a bounded sequence, for sufficiently large $n\geq N\in\N$ each gradient $\nabla^{g}t_{n}$ is timelike almost everywhere in $\Omega$. Moreover, coercivity of $\Fpq$ implies that the sequence $(t_{n})_{n\in\N}$ is bounded in $H^{p}(\Omega)$. By the sequential Banach-Alaouglu Theorem ($H^{p}(\Omega)$ is a reflexive Banach space) and weak closedness of $\Tspace$ (cf. Remark \ref{rem:properties}), there exists a subsequence $(t_{n_{k}})_{n_{k}\in\N}$ in $\Tspace$ which converges weakly to a function $t_{\alpha}\in \Tspace$. However, even if $t_{n_{k}}\in \Tt$ for each $n_{k}\geq N$, a priori, $\nabla^{g}t_{\alpha}$ could still be null or vanish a.e. on $\Omega$. Nevertheless, the assumption that $p\geq2$ rules out this possibility. 

    \par Indeed, as in the proof of the previous lemma, by the Rellich-Kondrachov theorem the bounded and weakly convergent subsequence $(t_{n_{k}})_{n_{k}\in\N}$ converges (after passing to a further subsequence) strongly in $H^{1}(\Omega)$ to $t_{\alpha}$ and also $|\nabla^{g}t_{n_{k}}|^{2}_{g}$ converges to $|\nabla^{g}t_{\alpha}|^{2}_{g}$ a.e. in $\Omega$. We now show that this actually implies that $\nabla^{g}t_{\alpha}$ must be timelike a.e. Seeking a contradiction, assume that the subset $A:=\{x\in\Omega:|\nabla^{g}t_{\alpha}|_{g}^{2}(x)=0\}$ is not a null set. Then, $(|\nabla^{g}t_{n_{k}}|^{-2q}_{g})_{k\in\N}$ is a sequence of measurable non-negative functions which diverge to $+\infty$ a.e. in $A$. Using Fatou's lemma
    \begin{align*}
        \liminf_{k\to\infty} F_{g,q}(t_{n_{k}})&\geq\liminf_{k\to\infty} \int_{A}\frac{1}{|\nabla^{g}t_{n_{k}}|_{g}^{2q}}d\mu_{h}\geq\int_{A}\liminf_{k\to\infty}\frac{1}{|\nabla^{g}t_{n_{k}}|_{g}^{2q}}d\mu_{h}=+\infty\;,
    \end{align*}
    which contradicts that $\lim_{k\to\infty}\Fpq(t_{n_{k}})=F_{\textrm{inf}}<\infty$ since $F_{g,q}(t_{n_{k}})\leq \alpha^{-1}\Fpq(t_{n_{k}})$. Hence, $t_{\alpha}\in\Tt$.
    
    Since $t_{n},t_{\alpha}\in \Tt$ for $n\geq N$ and $\Fpq$ is weakly lower semicontinuous on $\Tt$, it follows that $t_{\alpha}$ is a minimizer of the functional (i.e. $F_{\inf}=\Fpq(t_{\alpha})$),
    \begin{align}
        F_{\inf}\leq \Fpq(t_{\alpha})\leq \liminf_{k\to\infty} \Fpq(t_{n_{k}})=F_{\inf}\;.
    \end{align}
    Uniqueness of the minimizer then follows by strict convexity of $\Fpq$ on $\Tt$.

    In the case that $\alpha=0$, existence and uniqueness of a minimizer $t_{0}$ of $\Fpq$ is a direct application of the direct method in the calculus of variations using that $F_{h,p}$ is coercive, strictly convex and weakly lower semicontinuous on the set $\Tspace$ (cf. Lemma \ref{lem:props_functional}). In this case the arguments apply already for $p\geq1$. However, the gradient of the minimizer can be null or even vanish.
\end{proof}
Although temporal functions are time functions, functions with an a.e. past-directed timelike gradient (such as the minimizer $t_{\alpha}$ for $\alpha>0$ and $p\geq 2$) are not, a priori, a.e. steep nor generalized time functions (strictly increasing function along future-directed causal curves which are not necessarily continuous, cf. \cite[Definition 3.48]{Minguzzi_Sanchez}). For example, consider $n$-dimensional Minkowski spacetime (with coordinates $(x^{0},x^{1},\ldots,x^{n-1})$) and the function $f : \R^{1,n}\to \R$ with $f(x^{0},x^{1},\ldots,x^{n-1})=x^{0}$ for all $x\in \R^{1,n}\setminus\{0\}$ and $f(0)>0$. The gradient of $f$ is a.e. past-directed timelike but along the curve $\gamma :\R\to \R^{1,n},\gamma(s):=(s,0)$ the function $f$ is not strictly increasing.

Even if a unique $C^{k}$-regular (with $k\geq1$) minimizer $t_{\alpha}$ of the misalignment functional exists and its $g$-gradient $\nabla^{g}t_{\alpha}$ is timelike almost everywhere, $\nabla^{g}t$ may still be null or vanish on a measure zero subset. As a motivating example, consider the function $f :\R^{n}\to[0,\infty), x\mapsto \|x\|^{2}$, which is positive a.e., and the corresponding functional $F$ given by
\begin{align*}
    F(f):=\int_{\R^{n}}\frac{1}{(f(x))^{q}}d^{n}x\geq \int_{B_{\lambda}(0)}\frac{1}{(f(x))^{q}}d^{n}x\geq C\int_{0}^{\lambda}r^{n-1-2q}dr\;,
\end{align*}
with $\lambda\in(0,\infty)$. In particular, the above functional diverges for $q\geq \frac{n}{2}$. Recall that, in our setting, the value of the functional $\Fpq$ for the minimizer $t_{\alpha}$ is finite. Hence, it is to be expected that if $q$ is chosen large enough, this rules out the possibility that the gradient of the minimizer is null or vanishes on measure-zero subsets of $\Omega$.

The following proposition shows that if a minimizer of $\Fpq$ with $\alpha>0$ is sufficiently smooth and $q$ sufficiently large, then its gradient has to be everywhere timelike in $\Omega$.

\begin{Prp}\label{lem:everywhere_timelike}
    Let $t\in \Tspace$ with
    \begin{align*}
        F_{g,q}(t)=\int_{\Omega}\frac{1}{|\nabla^{g}t|_{g}^{2q}}d\mu_{h}<\infty\;.
    \end{align*}
    The gradient of $t$ is everywhere timelike in $\Omega$ if one of the following conditions holds:
    \begin{itemize}[leftmargin=2em]
        \item[i)] $p>\frac{n}{2}+1+\gamma$ and $q\geq \frac{n}{\gamma}$, with $\gamma\in(0,1)$
        \item[ii)] $p>\frac{n}{2}+2$ and $q\geq n$,
        \item[iii)] $p>\frac{n}{2}+3$, $q\geq \frac{n}{2}$ and there exists an $H^{p}$-regular extension of $t$ to an open subset $\Omega'$ with $\Omega\subset\Omega'$ such that $\nabla^{g}t$ is past-directed causal in $\Omega'$.
    \end{itemize}
\end{Prp}
\begin{proof}
    For $k\in\N$, $\gamma\in [0,1)$ and $p>\frac{n}{2}+k+\gamma$, it follows by the Sobolev embedding theorem that $t\in C^{k,\gamma}(\Omega)$. Assume that $\nabla^{g}t$ is null or vanishes at $x_{0}\in \Omega$ and consider a chart $(\varphi,U)$ around this point. We define the $C^{k-1,\gamma}$ regular function
    \begin{align}\label{eq:definition_f}
        f:= (-g(\nabla^{g}t,\nabla^{g}t))\circ\varphi^{-1} : \varphi(U')\subset \R^{n} \rightarrow [0,\infty)\;,
    \end{align} 
    where $U':=U\cap \Omega$, $f$ is non-negative because $\nabla^{g}t$ is a.e. timelike in $\Omega$ (and everywhere causal by continuity of $\nabla^{g}t$) and without loss of generality $y_{0}:=\varphi(x_{0})=0$, so $f(0)=0$. 
    \par Since the boundary $\partial\Omega$ is Lipschitz continuous (so, locally, $\Omega$ lies on one side of $\partial\Omega$), for any $x_{0}\in \Omega$ there exists a positive measure set $\Gamma\subseteq \mathbb{S}^{n-1}$ such that, for a sufficiently small $\lambda>0$, $(0,\lambda)\times \Gamma\subset \varphi(U')$. Then, for spherical coordinates $(r,\omega)\in (0,\lambda)\times\Gamma$
    \begin{align}\label{eq:bound_fbeta}
        f(r,\omega)\leq C r^{\beta} \implies \frac{1}{f(r,\omega)}\geq \frac{1}{Cr^{\beta}}\;\textrm{a.e. in}\; (0,\lambda)\times \Gamma\;,
    \end{align}
    where the inequality on the right-hand side is well-defined a.e. since $f$ can vanish at most on a null set. Then \eqref{eq:bound_fbeta} gives
    \begin{align*}
        &\int_{\varphi(U')}\frac{1}{(f(y))^{q}}d\mu_{h}\geq C\int_{(0,\lambda)\times \Gamma}\frac{1}{(f(y))^{q}}d^{n}y\geq C\int_{0}^{\lambda}\frac{r^{n-1}}{r^{\beta q}}dr
    \end{align*}
    where we used that $d\mu_{h}=\rho(y) d^{n}y\geq C d^{n}y$ in $\varphi(U')$ with $d^{n}y$ the Lebesgue measure on $\R^{n}$, $\rho$ positive and smooth and $C>0$ (note that in this proof constants change but are denoted with the same symbol). In particular, the above expression diverges if the exponent is smaller or equal to $-1$, i.e. whenever
    \begin{align}
        q\geq n/\beta\;.
    \end{align}
    Expression \eqref{eq:bound_fbeta} is satisfied in the different mentioned cases:
\begin{itemize}[leftmargin=2em]
    \item[(i)] If $p>n/2+1+\gamma$ with $\gamma\in(0,1)$, then $f\in C^{0,\gamma}(\varphi(U'))$. Hence,
    \begin{align}\label{eq:bound_f0}
        f(y)\leq C\|y\|^{\gamma}= Cr^{\gamma},\quad\textrm{with}\quad C\in(0,\infty)\;,
    \end{align}
    i.e. $f$ satisfies \eqref{eq:bound_fbeta} with $\beta=\gamma$ and the integral diverges for $q\geq n/\gamma$.
    \item[(ii)] If $p> n/2+2$, then $f$ is in particular locally Lipschitz continuous and expression \eqref{eq:bound_f0} with $\gamma=1$ holds. Hence, in this case the integral diverges\footnote{Alternatively a $C^{1,\gamma}$-extension of $f$ from $\varphi(U')$ to a neighbourhood of $y_{0}$ (since $f\in C^{1,\gamma}(\varphi(U'))$ and $\partial \Omega$ is Lipschitz, cf. \cite[Theorem 4.1]{Rychkov}) and a Taylor expansion of $f$ yield the same value of $\beta$.} for $q\geq n$.
\item[(iii)] Let $p>n/2+3$, $\nabla^{g}t$ be past-directed causal in an open subset $\Omega'$, with $\Omega\subset\Omega'$ and $U'$ a neighbourhood around $x_{0}$ with $U'\subset \Omega'$ and consider again the function $f$ given by \eqref{eq:definition_f}. Then, $y_{0}:=\varphi(x_{0})$ is a local minimum and $Df(y_{0})=0$, since $f$ is non-negative in $\varphi(U')$ ($\nabla^{g}t$ is causal in $\Omega'$), also if $x_{0}\in\partial\Omega$. The Taylor expansion of $f$ around $y_{0}=0$ yields
\begin{align*}
    f(y)&=\sum_{|\alpha|=2}\frac{1}{\alpha!}\frac{\partial^{\alpha}f}{\partial y^{\alpha}}(0)y^{\alpha}+R(y),\quad \textrm{with}\quad R(y)=o(\|y\|^{2}) \quad \textrm{as} \;y\to0\;.
\end{align*}
In spherical coordinates, for sufficiently small $\lambda>0$ and $\Gamma\subset \mathbb{S}^{n-1}$, 
\begin{align*}
    f(r,\omega) \leq Cr^{2}, \quad (r,\omega)\in(0,\lambda)\times \Gamma\;.
\end{align*}
Hence, $\beta=2$ and the integral diverges for $q\geq n/2$.

\end{itemize}

\end{proof}

In the following corollary we collect the different results obtained in this section. It implies that, choosing a sufficiently large Sobolev index $p$ and null penalizing index $q$, there exists a $C^{1,\gamma}$-regular temporal function which minimizes the alignment time function. We refer to this temporal function as the alignment time function.

\begin{Corollary}\label{coro:unique_minimizer} 
Let $\alpha\geq0$, $p,q\in \N$ and $\gamma\in(0,1)$. Assume that one of the following conditions is satisfied:
\begin{itemize}
    \item[i)] $\alpha>0$, $p>n/2+1+\gamma$ and $q\geq n/\gamma$.
    \item[ii)] $\alpha=0$ and $u$ is of $g$-gradient form.
\end{itemize}
Then, the minimizer $t_{\alpha}$ of the misalignment functional is a $C^{1,\gamma}(\Omega)$-regular temporal function and in the second case we have that $\nabla^{g}t_{0}=u$.
\begin{proof}
    That the gradient of the minimizer is everywhere past-directed timelike in $\Omega$, follows, for case $i)$, by Proposition \ref{lem:everywhere_timelike}, part $i)$.

    With respect to the second case, if $u$ is of $g$-gradient form, then there exists a smooth function $t\in\Tspace$ (cf. Remark \ref{rem:properties}, item ii)) such that $u=\nabla^{g}t$ and which minimizes $F_{h,p}=F_{p,q}^{0}$. Uniqueness of the minimizer implies that $t=t_{0}$, so $\nabla^{g}t_{0}$ is everywhere timelike. 
\end{proof}
\end{Corollary}
Of course, if desired, one can obtain everywhere timelikeness of the gradient of the minimizer for a lower value of $q$ by increasing the value of $p$ (recall cases $ii)$ and $iii)$ in Proposition \ref{lem:everywhere_timelike}).

\subsection{Elliptic regularity bootstrapping}\label{sec:elliptic_regularity}\hfill
\hfill
 \vspace{0.2cm}

In this subsection, the $C^{1,\gamma}$ regularity of the alignment time function is upgraded to smoothness. More specifically, the following proposition entails that, under the choice of $p$ and $q$ made at the end of the previous subsection (which guaranteed everywhere timelikeness of the gradient of the minimizer) the interior Euler-Lagrange equation associated to the variational problem are well-defined, uniformly elliptic and, in conclusion, the minimizer is actually smooth. Choosing the indices $p$ and $q$ large enough is crucial in order to ensure that the minimizer $t_{\alpha}$ is an interior point of the set $\Tspace$ (viewing $\Tspace$ as a subset of the set of functions in $H^{p}(\Omega)$ which satisfy the zero-mean condition) and thus that the Gateaux derivative of $F_{g,q}$ is well-defined. Note that this problem (Gateaux differentiability of a functional involving the Lorentzian norm on the set of past-directed causal vector fields) was already noted in \cite{octet} and is the reason that, in their work, they also consider one-sided variational derivatives.

\begin{Prp}\label{prop:smooth_minimizer}
    Let $\gamma\in(0,1)$, $q\geq n/\gamma$, $p> \frac{n}{2}+1+\gamma$ and $\alpha>0$. Then, the minimizer $t_{\alpha}\in \Tt$ of $\Fpq$ is a weak solution to the following Euler-Lagrange equation 
    \begin{align}\label{eq:strong_EL}
    \sum_{i=0}^{p-1}\nabla^{g\ast}(\nabla^{h\ast})^{i}(\nabla^{h})^{i}(\nabla^{g} t_{\alpha}-u)-\alpha q\;\div_{h}\Big(\frac{\nabla^{g} t_{\alpha}}{|\nabla^{g} t_{\alpha}|_{g}^{2q+2}}\Big)=0\;.
    \end{align}
    Moreover, it is a uniformly elliptic quasi-linear partial differential equation of order $2p$ and $t_{\alpha}\in C^{\infty}(\Omega^{\circ})$.
\end{Prp}
\begin{proof}
In order to determine the Euler-Lagrange equation, one has to compute the Gateaux derivative of $\Fpq$ in those directions $\eta$ for which $t_{\alpha}+s\eta$ is in the set $\Tspace$ for sufficiently small $|s|\in[0,\infty)$. In particular, this holds for any
$\eta\in H^{p}(\Omega)\supset \Tspace$ which satisfies the zero-mean condition since the gradient of the minimizer $t_{\alpha}$ is everywhere timelike and thus $t_{\alpha}$ is an interior point of $\Tspace$. Then, for any such direction $\eta$ the following variational equality holds
\begin{align}\label{eq:variational_eq}
    D\Fpq(t_{\alpha})[\eta]=0\;.
\end{align}
However, in order to view $t_{\alpha}$ as the weak interior solution to a partial differential equation, it is convenient to actually consider arbitrary test functions $\varphi\in C^{\infty}_{0}(\Omega^{\circ})$ (with $\Omega^{\circ}$ the topological interior of $\Omega$). This is possible, since for any such $\varphi$, the function $\eta_{\varphi}$ defined by
\begin{align}
    \eta_{\varphi}:=\varphi-\frac{1}{\mu_{h}(\Omega)}\int_{\Omega}\varphi \;d\mu_{h}\;,
\end{align}
satisfies the zero-mean condition and, as $\nabla^{g}\varphi=\nabla^{g}\eta_{\varphi}$, equation \eqref{eq:variational_eq} also holds for any $\varphi\in C^{\infty}_{0}(\Omega^{\circ})$. Thus we consider such test functions when computing the Gateaux derivatives.

The Gateaux derivative of $F_{h,p}$ in the direction of $\varphi\in C^{\infty}_{0}(\Omega^{\circ})$ is 
\begin{align}\label{eq:Gateaux_diff_Fhp}
    DF_{h,p}(t_{\alpha})[\varphi]&=\frac{d}{ds}\Big(\sum_{i=0}^{p-1}\int_{\Omega}h\big((\nabla^{h})^{i}(u-\nabla^{g} (t_{\alpha}+s\varphi),(\nabla^{h})^{i}(u-\nabla^{g} (t_{\alpha}+s\varphi)\big)\big)\Big|_{s=0}d\mu_{h} \nonumber\\
    &=-2\sum_{i=0}^{p-1}\int_{\Omega} h((\nabla^{h})^{i}\nabla^{g}\varphi,(\nabla^{h})^{i}(u-\nabla^{g} t_{\alpha}))d\mu_{h}\;,
\end{align}
which, integrating by parts, can also be rewritten in the distributional sense as
\begin{align}\label{eq:weak_EL}
        DF_{h,p}(t_{\alpha})[\varphi]=-2\sum_{i=0}^{p-1}\int_{\Omega} \varphi\;\nabla^{g\ast}(\nabla^{h\ast})^{i}(\nabla^{h})^{i}(u-\nabla^{g} t_{\alpha})d\mu_{h}\;,
\end{align}
where $\nabla^{g\ast}$ and $\nabla^{h\ast}$ denote the adjoints of $\nabla^{g}$ and $\nabla^{h}$ with respect to the Riemannian $h$-scalar product (cf. \cite[Section 2.2.2.1]{Petersen}). We now proceed analogously for $F_{g,q}$,
\begin{align}\label{eq:Lorentzian_functional}
    DF_{g,q}(t_{\alpha})[\varphi]&=2q\int_{\Omega}\frac{g(\nabla^{g} t_{\alpha},\nabla^{g} \varphi)}{|\nabla^{g} t_{\alpha}|_{g}^{2q+2}}d\mu_{h}=2q\int_{\Omega}\frac{h(\nabla^{g}t_{\alpha},\nabla^{h}\varphi)}{|\nabla^{g} t_{\alpha}|_{g}^{2q+2}}d\mu_{h}\nonumber\\
    &=-2q\int_{\Omega}\varphi\;\div_{h}\Big(\frac{\nabla^{g} t_{\alpha}}{|\nabla^{g} t_{\alpha}|_{g}^{2q+2}}\Big) d\mu_{h}\;,
\end{align}
where in the second step we used that $g(\nabla^{g}\varphi,X)=X(\varphi)=h(\nabla^{h}\varphi,X)$, with $X:=\nabla^{g}t_{\alpha}\in\Gamma(\Omega)$, and $\div_{h}(Y)=-(\nabla^{h})^{\ast}Y$ for $Y\in \Gamma(\Omega)$. Moreover, $|\nabla^{g}t_{\alpha}|_{g}^{-2q-2}$ is integrable because $\nabla^{g}t_{\alpha}$ is everywhere timelike and, since $\Omega$ is compact, the norm $|\nabla^{g}t_{\alpha}|^{2}_{g}$ can be bounded away from zero. Hence, for any $\varphi\in C^{\infty}_{0}(\Omega^{\circ})$ the following weak Euler-Lagrange equation holds:
\begin{align}
    \int_{\Omega} \varphi\Big(\sum_{i=0}^{p-1}\nabla^{g\ast}(\nabla^{h\ast})^{i}(\nabla^{h})^{i}(\nabla^{g} t_{\alpha}-u)-\alpha q\;\div_{h}\Big(\frac{\nabla^{g} t_{\alpha}}{|\nabla^{g} t_{\alpha}|_{g}^{2q+2}}\Big)\Big) d\mu_{h}=0\;.
\end{align}
I.e. $t_{\alpha}$ is a weak solution to the following quasi-linear partial differential equation
\begin{align}\label{eq:PDE}
    \sum_{i=0}^{p-1}L_{h,i}t_{\alpha}-\alpha qL_{g}t_{\alpha}=f(u)\;,
\end{align}
in $\Omega^{\circ}$, where $L_{h,i}$ and $L_{g}$ are differential operators acting on the minimizer $t_{\alpha}$ and $f(u)$ a smooth function depending on the vector field $u$
\begin{align}
    &\begin{cases}
        L_{h,i}t:=D_{i}^{\ast}D_{i}t\\
        L_{g}t:=\div_{h}(|\nabla^{g}t|_{g}^{-2q-2}\nabla^{g}t)
    \end{cases}
    &f(u):=\sum_{i=0}^{p-1}\nabla^{g\ast}(\nabla^{h\ast})^{i}(\nabla^{h})^{i}u
\end{align}
with $D_{i}:=(\nabla^{h})^{i}\nabla^{g}$. The above PDE has order $2p$ and $L_{h,p-1}=D^{\ast}_{p-1}D_{p-1}$ is its higher order term ($L_{g}$ is only of second order). In local coordinates, $L_{h,p-1}$ is
\begin{align}\label{eq:higher_order_operator}
    &(D_{p-1}t)^{i}_{j_{1}\ldots j_{p-1}}=g^{ik}\partial_{j_{1}}\ldots \partial_{j_{p-1}}\partial_{k}t +\textrm{(lower order terms)}\nonumber
     \\
     &L_{h,p-1}t\!=\!(-1)^{p}g^{lk}g^{mi}h_{li}h^{i_{1}j_{1}}\ldots h^{i_{p-1}j_{p-1}} \partial_{k}\partial_{m}\partial_{i_{1}}\ldots \partial_{i_{p-1}}\partial_{j_{1}} \ldots \partial_{j_{p-1}}t +\textrm{(l.o.t)}
\end{align}
where the adjoint $D_{p-1}^{\ast}$ was computed using local coordinates and integrating by parts $p$-times against a function $\varphi\in C^{\infty}_{0}(\Omega)$:
\begin{align*}
    &\int_{\Omega}\varphi\;L_{h,p-1}t\;d\mu_{h}=\int_{\Omega}\varphi\;D_{p-1}^{\ast}D_{p-1}t\;d\mu_{h}=\int_{\Omega}h(D_{p-1}t,D_{p-1}\varphi)\;d\mu_{h}\\
    &=\int_{\Omega} h_{li}h^{i_{1}j_{1}}\ldots h^{i_{p-1}j_{p-1}} (D_{p-1}t)^{l}_{i_{1} \ldots i_{p-1}}g^{ik}\partial_{j_{1}} \ldots \partial_{j_{p-1}}\partial_{k}\varphi\;d\mu_{h} +\textrm{(l.o.t.)}\\
    &=(-1)^{p}\int_{\Omega} \varphi\;h_{li}g^{ik}h^{i_{1}j_{1}}\ldots h^{i_{p-1}j_{p-1}} \partial_{j_{1}} \ldots \partial_{j_{p-1}}\partial_{k}(D_{p-1}t)^{l}_{i_{1} \ldots i_{p-1}}\;d\mu_{h} +\textrm{(l.o.t.)}\;.
\end{align*}
From expression \eqref{eq:higher_order_operator} the principal symbol (using the conventions in \cite[Section 7A]{Folland}) for the Euler-Lagrange equation satisfies that
\begin{align*}
    (-1)^{p}\sigma_{2p}(L_{h,p-1})(\xi)&=g^{lk}g^{mi}h_{li}h^{i_{1}j_{1}}\ldots h^{i_{p-1}j_{p-1}} \xi_{k}\xi_{m}\xi_{i_{1}}\ldots \xi_{i_{p-1}}\xi_{j_{1}} \ldots \xi_{j_{p-1}}\\
    &=|\xi|_{h}^{2p-2}h(\xi^{\sharp,g},\xi^{\sharp,g})\geq C|\xi|_{h}^{2p}\;.
\end{align*}
with $\xi\in T_{x}^{\ast}\Omega$, $(\xi^{\sharp,g})^{i}:=g^{ij}\xi_{j}$ and, as $h_{x}(\xi^{\sharp,g}(x),\xi^{\sharp,g}(x))=(g^{lk}g^{mi}h_{li}\xi_{k}\xi_{m})(x)$ is positive definite on $T^{\ast}_{x}\Omega$ (for each $x\in \Omega^{\circ}$) there exists a constant $C(x)>0$ such that $h(\xi^{\sharp,g},\xi^{\sharp,g})(x)\geq C(x)|\xi|_{h}^{2}(x)$ and which can be chosen independent of the base-point $x\in \Omega$ by compactness of $\Omega$. Thus, the Euler-Lagrange equation is uniformly elliptic. In order to apply interior regularity theory, we rewrite the elliptic equation as
\begin{align*}
    Lt_{\alpha}:=\sum_{i=0}^{p-1} L_{h,i}t_{\alpha}=\tilde{f}\;,
\end{align*}
where $\tilde{f}:=f(u)+\alpha qL_{g}t_{\alpha}$. In particular, since $\nabla^{g}t_{\alpha}$ is everywhere timelike in the compact set $\Omega$ and $p-1>n/2$ (so the Sobolev multiplication theorems apply), it follows that $|\nabla^{g}t_{\alpha}|_{g}^{-2q-2}\nabla^{g}t_{\alpha}\in H^{p-1}(\Omega)$ and so $\tilde{f}\in H^{p-2}(\Omega)$. Since $L$ is a uniformly elliptic operator of order $2p$ and $Lt_{\alpha}\in H^{p-2}(\Omega)$, the $L^{2}$ based interior elliptic regularity theory implies that $t_{\alpha}\in H_{\textrm{loc}}^{3p-2}(\Omega)$ (cf. \cite[Theorem 3.2]{Lions_Magenes}). An iteration of this argument and the Sobolev embedding theorems yield interior smoothness of the minimizer $t_{\alpha}$ because, for arbitrary $r\in\N$, if $Lt_{\alpha}\in H^{r-2}_{\textrm{loc}}(\Omega)$, then $t_{\alpha}\in H_{\textrm{loc}}^{2p+r-2}(\Omega)$.

 \end{proof}

\begin{Remark}\label{rem:p-dalembertian}\rm{
In \cite{octet} the $p$-d'Alembertian (defined earlier in \cite{Mondino_Suhr} under the name of \emph{$q$-Box operator}) is defined as the following operator
\begin{align}
    \Box_{p}t:=-\div_{g}(|\nabla^{g}t|_{g}^{p-2}\nabla^{g}t)\;.
\end{align}
Since our Riemannian metric $h$ is given by \eqref{eq:Riemannian_metric}, \cite[Corollary 2.4]{Olea} implies that $\div_{g}(X)=\div_{h}(X)$ for any $X\in \Gamma(TM)$. Thus, the non-linear term appearing in the interior Euler-Lagrange equation is, up to a multiplicative constant, the $r$-d'Alembertian with $r:=-2q<0$.
}\end{Remark}

It is clear that, by the elliptic regularity arguments used in the previous proposition, also for $\alpha=0$ the minimizer $t_{0}$ is smooth. Moreover, in the previous proposition we get smoothness of the minimizer in the interior of $\Omega$ but not up to the boundary since we have not prescribed boundary conditions (and, as discussed in Remark \ref{rem:properties} v), it is preferable for our construction to not use boundary conditions). Furthermore, that the gradient of the minimizer of $\Fpq$ is timelike everywhere plays a very important role in proving its smoothness since, if this were not the case, we would not obtain an Euler-Lagrange equation but rather a differential inequality.

Whereas existence and uniqueness of a minimizer only requires $p\geq 2$, the derivation of the Euler-Lagrange equation relies on $p>n/2+1$ and $q\geq n/\gamma$ (with $\gamma\in(0,1)$). The reason is that these conditions guarantee that the minimizer is an interior point of the constrained set $\Tspace$. Nevertheless, it would be interesting to analyze whether it is possible to obtain the Euler-Lagrange equation for a lower value of $p$. Note that this would still require that the minimizer $t_{\alpha}$ is an interior point of $\Tspace$ and integrability of the problematic term
\begin{align*}
    \nabla^{g}t_{\alpha}|\nabla^{g}t_{\alpha}|_{g}^{-2-2q}\,.
\end{align*}

Even though the Euler-Lagrange equation has a rather complicated form it can still be used to construct interior minimizers of the misalignment functional (meaning that its first Gateaux derivative in the direction of an arbitrary $\varphi\in C^{\infty}_{0}(\Omega^{\circ})$ vanishes) for specific simple examples of spacetimes and vector fields. In particular, it can be used to verify whether, for such examples, a specific ansatz for the alignment time function is an interior minimizer or not.

\begin{Example}\label{ex:example_Minkowski}
\rm{
    Consider $n$-dimensional Minkowski spacetime $(\R^{1,n},\eta)$ with coordinates $(x^{0},x^{1},\ldots,x^{n-1})\in\R^{n}$, a compact subset $\Omega\subset\R\times( \R^{n-1}\setminus\{0\})$ and the following smooth past-directed timelike vector field
    \begin{align*}
        u=-f(r)\partial_{x^{0}}\;,\quad\textrm{with}\quad r:=\sqrt{(x^{1})^{2}+\ldots+(x^{n-1})^{2}}\;,
    \end{align*}
    where $f: (0,\infty)\to(0,\infty)$ is smooth. Unless $f$ is constant, $u$ is not of gradient form:
    \begin{align*}
        &u^{\flat}=f(r)dx^{0},\quad d(u^{\flat})=f'(r)dr\wedge dx^{0}\;.
    \end{align*}
    Hence, $u^{\flat}$ is closed if and only if $f$ is constant. It follows that, for non-constant $f$, $u^{\flat}$ cannot be of gradient form. 
    \par On the other hand, the Riemannian metric $h$ is simply the Euclidean metric $\delta_{\mathbb{R}^{n}}$
    \begin{align*}
        h=\eta+2\frac{u^{\flat}\otimes u^{\flat}}{|u|^{2}_{\eta}}=dx^{0}\otimes dx^{0}+\sum_{i=1}^{n-1}dx^{i}\otimes dx^{i}=\delta_{\mathbb{R}^{n}}\implies d\mu_{h}=d^{n}x\;.
    \end{align*}
    Consider the following ansatz for the gradient of the alignment time function
    \begin{align*}
        \nabla^{\eta}t_{c}=-a(r)\partial_{x^{0}}\;,
    \end{align*}
    with $a : (0,\infty)\to(0,\infty)$ smooth. In particular, it follows that $a$ must actually be constant (since $\nabla^{\eta}t_{c}$ has no radial component, $\partial_{r}t_{c}=0$ and thus $\partial_{x^{0}}t_{c}=a(r)$ is actually $r$-independent). So, the considered time function is:
    \begin{align}\label{eq:ansatz}
        t_{c}=cx^{0}+C_{c}\;,\quad\textrm{with} \quad C_{c}=-\frac{c}{\mu_{h}(\Omega)}\int_{\Omega}x^{0} \;d^{n}x\;,
    \end{align}
    with $c$ and $C_{c}$ constants and the latter guarantees that $t_{c}$ satisfies the zero-mean condition. We now show that $t_{c}$ satisfies the interior Euler-Lagrange equation from Proposition \ref{prop:smooth_minimizer} although it does not fix the value of $c$.
    
    In the first place, since $\partial_{x^{0}}$ is parallel with respect to the Euclidean metric $h$, applying iteratively the connection $\nabla^{h}$ to $\nabla^{\eta}t_{c}-u=b(r)\partial_{x^{0}}$ (where $b(r):=f(r)-c$) yields the following simple expression
    \begin{align*}
        &\nabla^{h}_{x^{\mu}}(b(r)\partial_{x^{0}})=\partial_{x^{\mu}}(b(r))\partial_{x^{0}}\implies \nabla^{h}(b(r)\partial_{x^{0}})=db\otimes \partial_{x^{0}}\\
        &(\nabla^{h})^{i}(b(r)\partial_{x^{0}})=((\nabla^{h})^{i}b)\otimes \partial_{x^{0}}\;,
    \end{align*}
    where $i\in\N$ and $\mu\in\{1,\ldots,n\}$. Applying now the adjoint operator $\nabla^{h\ast}$ $i$-times gives back a vector collinear to $\partial_{x^{0}}$ which vanishes after acting with $\nabla^{\eta\ast}$:
    \begin{align*}
        &(\nabla^{h\ast})^{i}(\nabla^{h})^{i}(b(r)\partial_{x^{0}})=((\nabla^{h\ast})^{i}(\nabla^{h})^{i}(b))\partial_{x^{0}}=:\tilde{b}(r)\partial_{x^{0}}\\
        &\nabla^{\eta\ast}(\nabla^{h\ast})^{i}(\nabla^{h})^{i}(b(r)\partial_{x^{0}})=-\div_{\eta}(\tilde{b}(r)\partial_{x^{0}})=-\partial_{x^{0}}(\tilde{b}(r))=0\;.
    \end{align*}
    Hence, the $\nabla^{\eta}t_{c}-u$ term appearing in the Euler Lagrange equation \eqref{eq:strong_EL} vanishes. It remains to prove that also the remaining term vanishes. Note that
    \begin{align*}
        \frac{\nabla^{\eta}t_{c}}{|\nabla^{\eta}t_{c}|_{\eta}^{2q+2}}=-c^{-2q-1}\partial_{x^{0}}\;,
    \end{align*}
    so the $h$-divergence of this vector vanishes. So, for any constant $c\in(0,\infty)$, the time function $t_{c}=cx^{0}+C_{c}$, with $C_{c}$ fixed in \eqref{eq:ansatz}, satisfies the interior Euler-Lagrange equation. Therefore, the one-parameter family of time functions $\{cx^{0}+C_{c}:c>0\}$ are interior critical points of the misalignment functional, in the sense that for all $\varphi\in C^{\infty}_{0}(\Omega^{\circ})$ and $c>0$ it holds that
    \begin{align}
        D\Fpq(t_{c})[\varphi]=0\;.
    \end{align}
    Although $\{cx^{0}+C_{c}:c>0\}$ is a one-parameter family of interior critical points of the misalignment functional, they are even interior minimizers by convexity of the functional. However, this does not directly yield a global minimizer of $\Fpq$ (i.e. the first variation need not vanish for arbitrary $\eta\in H^{p}(\Omega)$ satisfying the zero mean condition). This is a consequence of the integration by parts step used in the derivation of the interior Euler-Lagrange equation: this step exploited that $\varphi|_{\partial\Omega}=0$ and thus expressions \eqref{eq:weak_EL} and \eqref{eq:Lorentzian_functional} need not hold for an arbitrary $\eta\in H^{p}(\Omega)$. If one would replace the zero-mean condition in the definition of $\Tspace$ with trace conditions for $(\nabla^{h})^{j}t$ for each $j\in\{0,\ldots,p-1\}$, then the solution to the Euler-Lagrange equation (with boundary conditions coming from the trace conditions) would be the global minimizer (by convexity), since, in that case, for any admissible variation $\eta\in H^{p}(\Omega)$, $(\nabla^{h})^{j}\eta$ vanishes on the boundary $\partial\Omega$ for any $j\in\{0,\ldots, p-1\}$. Nevertheless, this would not imply that, for the considered example, the global minimizer is of the form $t_{c}=cx^{0}+C_{c}$ since such a function does not necessarily satisfy the prescribed boundary conditions.
    
    One can still derive the optimal value of $c$ for which $t_{c}=cx^{0}+C_{c}$ minimizes the functional $\Fpq$ over the considered family of interior minimizers $\{cx^{0}+C_{c}:c>0\}$. It can be determined by evaluating the functional $\Fpq$ on the family of interior minimizers, which then yields a function of the variable $c$ and can again be minimized
    \begin{align}\label{eq:minimum_c}
    \Fpq(t_{c})&=\int_{\Omega}\big((c-f)^{2}+\alpha c^{-2q}\big)d^{n}x +\sum_{i=1}^{p-1}\|(\nabla^{h})^{i}u\|^{2}_{L^{2}(\Omega)}\nonumber\\
        &=c_{2}c^{2}-2c_{3} c+\alpha c_{2}c^{-2q}+c_{1}\\
        &\textrm{with}\quad c_{1}:=\sum_{i=1}^{p-1}\|(\nabla^{h})^{i}u\|^{2}_{L^{2}(\Omega)}+\int_{\Omega}f^{2}d^{n}x,\; c_{2}:=\mu_{h} (\Omega),\;c_{3}:=\int_{\Omega}fd^{n}x\nonumber\;.
    \end{align}
    where we used that $(\nabla^{h})^{j}(\nabla^{\eta}t_{c})=0$ for $j\geq1$. The optimal interior minimizer of $\Fpq$ is then the time function $t_{c}=cx^{0}+C_{c}$ for which the constant $c$ minimizes \eqref{eq:minimum_c} (a function of $c$). We will denote that value of the constant by $c_{\textrm{min}}$. Computing the first and second derivative (with respect to $c$) of the function \eqref{eq:minimum_c} yields
    \begin{align}\label{eq:eq_for_c}
        & 2c_{2}c_{\textrm{min}}-2c_{3}-2\alpha c_{2} qc_{\textrm{min}}^{-2q-1}=0\\
        &2c_{2}+2\alpha q(2q+1) c_{2} c_{\textrm{min}}^{-2q-2}>0\nonumber\;,
    \end{align}
    where positivity of the second equation implies that the critical value of $c$ of equation \eqref{eq:minimum_c} is a minimum. Hence, $t_{c_{\textrm{min}}}=c_{\textrm{min}}x^{0}+C_{c_{\textrm{min}}}$, with $c_{\textrm{min}}$ the solution to \eqref{eq:eq_for_c} and $C_{c_{\textrm{min}}}$ given by \eqref{eq:ansatz}, is the optimal interior minimizer of $\Fpq$. Note also that $c_{\textrm{min}}$ depends on $q,\alpha,f$ and $\Omega$ but not on $p$ (the constant $c_{1}$ was the only one which depended on $p$ and it does not appear in \eqref{eq:eq_for_c}).
}\end{Example}

\section{Further properties of the alignment time function}\label{sec:properties}
The previous section establishes existence, uniqueness and smoothness of the alignment time function but is rather non-constructive (the Euler-Lagrange equation presents a complicated form). Hence, a priori, it is not clear which properties this temporal function satisfies. In this section, different important features of the alignment time function are derived.

\subsection{Improved steepness: the $q\to\infty$ limit}\label{subsec:limit}\hfill
\vspace{0.2cm}

The previous section established that, if the parameters $p,q\in\N$ are chosen sufficiently large, there exists a constant $c>0$ (dependent on the parameters $p$, $q$, the vector field $u$, the metric $g$ and the set $\Omega$) such that
\begin{align}
    |\nabla^{g}t_{\alpha}|_{g}(x)\geq c(p,q)\quad \forall x\in\Omega\;.
\end{align}
However, in general, one cannot estimate this constant $c$ and thus neither the steepness of the gradient of the alignment time function. Intuitively one would expect such an estimate to exist for a sufficiently large penalization index $q$. The following proposition formalizes this idea.

In this section, in order to explicitly emphasize on the dependence on the parameters $p$ and $q$, the minimizer of the misalignment functional $F_{p,q}$ for a fixed $p,q\in \N$ will be denoted by $t_{p,q}$.

\begin{Prp}\label{prop:q_infinity}
    Let $p\geq2$ and $\alpha>0$ be fixed, and $t_{p,q}$ be the unique minimizer of $F^{\alpha}_{p,q}$. Then, the sequence $(t_{p,q})_{q\in\N}$ converges to $t_{p,\infty}\in \Tt$ as $q\to\infty$ strongly in $H^{1}(\Omega)$ and weakly in $H^{p}(\Omega)$. Along this sequence,
    \begin{align}\label{eq:lim_sup_bound}
        \limsup_{q\to\infty}\underset{x\in\Omega}{\text{\rm{ess inf}}}\;\big(|\nabla^{g}t_{p,q}|_{g}(x)\big)\leq \underset{x\in\Omega}{\text{\rm{ess inf}}}\;\big(|\nabla^{g}t_{p,\infty}|_{g}(x)\big)\;,
    \end{align}
    with the following bound
    \begin{align}\label{eq:ess_inf_bound}
        \underset{x\in \Omega}{\text{\rm{ess inf}}}\;(|\nabla^{g}t_{p,\infty}|_{g}(x))\geq1\;.
    \end{align}
    Moreover, $t_{p,\infty}$ is the unique minimizer of $F_{h,p}$ on the set
    \begin{align}
        \mathcal{S}^{p}:=\{t\in \Tspace : \underset{x\in\Omega}{\text{\rm{ess inf}}}\; |\nabla^{g}t|_{g}(x)\geq 1\}\;.
    \end{align}
    Finally, if $p> n/2+1$ we have that
    \begin{align}
        \inf_{x\in\Omega}|\nabla^{g}t_{p,q}|_{g}(x)\to \inf_{x\in\Omega}|\nabla^{g}t_{p,\infty}|_{g}(x)\geq1 \quad \textrm{as }q\to\infty\;.
    \end{align}
\end{Prp}
\begin{proof}
    Let $p\geq2$. Since $t_{p,q}$ is a minimizer of $F^{\alpha}_{p,q}$ and $(M,g)$ is stably causal there exists a temporal function $\tau$ with (possibly after a reparametrization) $|\nabla^{g}\tau|_{g}\geq 1$ and a $q$-independent constant $C\in(0,\infty)$ such that
    \begin{align}\label{eq:uniform_bound}
        F^{\alpha}_{p,q}(t_{p,q})\leq F^{\alpha}_{p,q}(\tau)\leq C\;.
    \end{align}
    Boundedness of the real-valued sequence $(F^{\alpha}_{p,q}(t_{p,q}))_{q\in\N}$ implies, by coercivity of the misalignment functional, that the sequence $(t_{p,q})_{q\in\N}$ is bounded in the $H^{p}(\Omega)$-norm. By the Banach-Alaoglu theorem and weak closedness of $\Tspace$, there exists a subsequence $(t_{p,q_{k}})_{k\in\N}$ converging weakly to a limit function $t_{p,\infty}\in\Tspace$. By the Rellich-Kondrachov theorem (after passing to a further subsequence) $(t_{p,q_{k}})_{k\in\N}$ converges also strongly in $H^{1}(\Omega)$, which implies a.e. convergence of (a subsequence of) $|\nabla^{g}t_{p,q_{k}}|_{g}$ to $|\nabla^{g}t_{p,\infty}|_{g}$.\\
    In order to prove expression \eqref{eq:lim_sup_bound}, consider the following functional
    \begin{align}\label{eq:ess_inf_def}
        E(t):=\underset{x\in\Omega}{\text{\rm{ess inf}}}\;|\nabla^{g}t|_{g}(x)\;,
    \end{align}
    with $t\in\Tspace$. Let $\varepsilon>0$. By the definition of the essential infimum the following set
    \begin{align}
        A_{\varepsilon}:=\{x\in\Omega:|\nabla^{g}t_{p,\infty}|_{g}(x)\leq E(t_{p,\infty})+\varepsilon\}
    \end{align}
    has positive measure. Since $|\nabla^{g}t_{p,q_{k}}|_{g}\to|\nabla^{g}t_{p,\infty}|_{g}$ a.e. in $A_{\varepsilon}$, by the Egorov theorem, there exists a positive measure subset $B_{\varepsilon}\subset A_{\varepsilon}$ in which the convergence is uniform. Thus, using the triangle inequality (for the absolute value) there exists a $K\in \N$ such that for all $x\in B_{\varepsilon}$ and $k\geq K$
    \begin{align}\label{eq:essinf_estimate}
        |\nabla^{g}t_{p,q_{k}}|_{g}(x)&\leq\big|\;|\nabla^{g}t_{p,q_{k}}|_{g}(x)-|\nabla^{g}t_{p,\infty}|_{g}(x)\big|+|\nabla^{g}t_{p,\infty}|_{g}(x)\leq E(t_{p,\infty})+2\varepsilon
    \end{align}
    Then, since $\mu_{h}(B_{\varepsilon})>0$ for $k\geq K$ it holds that
    \begin{align*}
        E(t_{p,q_{k}})\leq E(t_{p,\infty})+2\varepsilon \implies \limsup_{k\to\infty} E(t_{p,q_{k}})\leq E(t_{p,\infty})+2\varepsilon\;,
    \end{align*}
    taking the $\varepsilon\to0$ limit gives back expression \eqref{eq:lim_sup_bound} for the subsequence $(t_{p,q_{k}})_{q_{k}\in\N}$. Once convergence of  $(t_{p,q})_{q\in\N}$ is proven, the estimate holds for the full sequence.
    
    Secondly, we prove bound \eqref{eq:ess_inf_bound}. Assume that $E(t_{p,\infty})<1$ and choose a $\varepsilon>0$ small enough such that $E(t_{p,\infty})+2\varepsilon<1$. From expression \eqref{eq:essinf_estimate} it then follows that
    \begin{align}\label{eq:limsup_2}
        \!\!\!F_{g,q_{k}}(t_{p,q_{k}})=\int_{\Omega}\frac{d\mu_{h}}{|\nabla^{g}t_{p,q_{k}}|_{g}^{2q_{k}}}\! \geq\! \int_{B_{\varepsilon}}\frac{1}{|\nabla^{g}t_{p,q_{k}}|_{g}^{2q_{k}}}d\mu_{h}\geq \mu_{h}(B_{\varepsilon}) (E(t_{p,\infty})+2\varepsilon)^{-2q_{k}}\! ,
    \end{align}
    which diverges as $k\to \infty$, contradicting \eqref{eq:uniform_bound}.
    
    \noindent In the third place, we show convergence of the sequence of minimizers $(t_{p,q})_{q\in\N}$. Note that $\mathcal{S}^{p}$ is a convex set because of the reversed triangle inequality. Let $t,t'\in \mathcal{S}^{p}$ be arbitrary and $\lambda\in(0,1)$. Then, $t_{\lambda}:=\lambda t'+(1-\lambda)t\in \mathcal{S}^{p}$ since
    \begin{align*}
        &|\nabla^{g}t_{\lambda}|_{g}\geq \lambda|\nabla^{g}t'|_{g}+(1-\lambda)|\nabla^{g}t|_{g}\implies E(t_{\lambda})\geq \lambda E(t')+(1-\lambda)E(t)\geq1\;.
    \end{align*}
    In particular, consider the case with $|\nabla^{g}t'|_{g}>1$, for which $E(t_{\lambda})>1$. Then, $F_{g,q_{k}}(t_{\lambda})\to0$ since $F_{g,q_{k}}(t_{\lambda})\leq \mu_{h}(\Omega)(E(t_{\lambda}))^{-2q_{k}}\to0$ as $k\to\infty$). Using that $t_{p,q}$ minimizes $\Fpq$ it follows that:
    \begin{align}\label{eq:limsup_Fhp}
        &\limsup_{k\to\infty}F_{h,p}(t_{p,q_{k}})\leq \limsup_{k\to\infty}(F_{h,p}(t_{\lambda})+\alpha F_{g,q_{k}}(t_{\lambda}))= F_{h,p}(t_{\lambda})\nonumber\\
        &\implies \limsup_{k\to\infty}F_{h,p}(t_{p,q_{k}})\leq F_{h,p}(t)\;,
    \end{align}
    where we used that, $t_{\lambda}\to t$ as $\lambda\to0$. Using weak lower semicontinuity of $F_{h,p}$ yields
    \begin{align*}
        F_{h,p}(t_{p,\infty})\leq \liminf_{k\to\infty} F_{h,p}(t_{p,q_{k}})\leq\limsup_{k\to\infty} F_{h,p}(t_{p,q_{k}})\leq F_{h,p}(t)\;.
    \end{align*}
    Since $t\in \mathcal{S}^{p}$ was arbitrary and $t_{p,\infty}\in \mathcal{S}^{p}$, we conclude that
    \begin{align}
        F_{h,p}(t_{p,\infty})=\min_{t\in \mathcal{S}^{p}} F_{h,p}(t)\;,
    \end{align}
    i.e. $t_{p,\infty}$ is the minimizer of $F_{h,p}$ on the set $\mathcal{S}^{p}$. By strict convexity of $F_{h,p}$, the minimizer is unique. Hence, any subsequence of the $H^{p}(\Omega)$-bounded sequence $(t_{p,q})_{q\in\N}$ admits a further subsequence converging weakly in $H^{p}(\Omega)$ and strongly in $H^{1}(\Omega)$ (by Banach-Alaoglu and Rellich-Kondrachov) to the unique minimizer $t_{p,\infty}$ of $F_{h,p}$ on $\mathcal{S}^{p}$. This implies convergence of the full sequence $(t_{p,q})_{q\in\N}$.

    Finally, if $p>n/2+1$ then the bounded (in $H^{p}(\Omega)$) sequence $(t_{p,q})_{q\in\N}$ has a subsequence $(t_{p,q_{k}})_{q_{k}\in\N}$ converging in $C^{1}(\Omega)$ to $t_{p,\infty}\in C^{1}(\Omega)$ (that the limit has to be $t_{p,\infty}$ follows from convergence of $(t_{p,q})_{q\in\N}$ and uniqueness of limits). Since this holds for any subsequence, the full sequence $(t_{p,q})_{q\in\N}$ converges in $C^{1}$, so the essential infimums in \eqref{eq:lim_sup_bound} become infimums. Then, $C^{0}$-convergence of $|\nabla^{g}t_{p,q}|_{g}$ and the bound
    \begin{align*}
        \big| \inf_{x\in\Omega}|\nabla^{g}t_{p,q}|_{g}-\inf_{x\in\Omega}|\nabla^{g}t_{p,\infty}|_{g}\big|\leq  \||\nabla^{g}t_{p,q}|_{g}-|\nabla^{g}t_{p,\infty}|_{g}\|_{C^{0}(\Omega)}\;,
    \end{align*}
    imply convergence of the sequence $\big(\inf_{x\in\Omega}|\nabla^{g}t_{p,q}|_{g}\big)_{q\in\N}$ to the infimum of $\nabla^{g}t_{p,\infty}$.
\end{proof}
In specific cases, the targeted steepness of $t_{p,\infty}$ (i.e. that $E(t_{p,\infty})\geq 1$) might not be the most interesting one. In that case, one can consider instead the following null penalization functional
 \begin{align*}
     F_{g,q}^{C}(t):=\int_{\Omega}\frac{C^{2q}}{|\nabla^{g}t|_{g}^{2q}}d\mu_{h}\;,
 \end{align*}
 with $C>0$ the desired steepness bound. Then, the previous proposition implies that $E(t_{p,\infty})\geq C$. In Remark \ref{rem:bounded_gradient} this is used in order to show that, after multipliying penalization functional with a suitable constant, it follows that the gradient of $t_{p,\infty}$ is steeper than $u$.

\begin{Remark}\label{rem:tpinfty_isotone}\rm{
 The previous proposition implies that $t_{p,\infty}$ is almost everywhere steep for any $p\geq2$, a property which may not be satisfied by an arbitrary minimizer $t_{p,q}$ with $p\geq 2$ and $q\in \N$. In general, $t_{p,\infty}$ may fail to be a causal or isotone function (i.e. a function which is non-decreasing along future-directed causal curves). Nevertheless, if $p>n/2$ then $t_{p,\infty}$ is also continuous and, by \cite[Theorem 1.28]{Minguzzi_causality}, it is a causal function and if $p\geq n/2+1$ then it is a temporal function in $\Omega$ (with steepness constant $C\geq 1$).
}\end{Remark}

In the following corollary we use that $H^{p}(\Omega)$ is an inner product space in order to upgrade strong $H^{1}(\Omega)$ convergence of the sequence $(t_{p,q})_{q\in\N}$ to strong $H^{p}(\Omega)$ convergence. 

Furthermore, provided that $t_{p,\infty}$ is an interior point of the set $\mathcal{S}^{p}$ (viewing it again as a subset of the set of functions in $H^{p}(\Omega)$ which satisfy the zero-mean condition), we show that this temporal function is actually smooth. Recall that Proposition \ref{prop:smooth_minimizer}, which established smoothness of the minimizer $t_{\alpha}$, required that $p>n/2+1+\gamma$ and $q\geq n/\gamma$ in order to guarantee that $t_{\alpha}$ was an interior point of $\Tspace$. Otherwise, $t_{\alpha}$ would only have satisfied a differential inequality and not necessarily the weak Euler-Lagrange equation. The same problem appears with the minimizer $t_{p,\infty}$ of $F_{h,p}$ over the constrained set $\mathcal{S}^{p}$. In particular, the condition $E(t_{p,\infty})>1$ is necessary for the weak Euler-Lagrange equation to be well-defined.

\begin{Corollary}\label{cor:Hp_convergence}
    Let $p\geq2$ and $\alpha>0$:
    \begin{itemize}[leftmargin=2em]
        \item[(i)] $(F_{h,p}(t_{p,q}))_{q\in\N}$ converges to $F_{h,p}(t_{p,\infty})$ and $(t_{p,q})_{q\in\N}$ converges strongly in $H^{p}(\Omega)$ to $t_{p,\infty}\in\mathcal{S}^{p}$.
        \item [(ii)] If $p> n/2+1$ and
    \begin{align}\label{eq:infimum_tpinfty}
        \inf_{x\in\Omega}\;(|\nabla^{g}t_{p,\infty}|_{g}(x))>1  \;.
    \end{align}
    Then, $t_{p,\infty}\in C^{\infty}(\Omega^{\circ})$.
    \end{itemize}
\end{Corollary}
\begin{proof}
    Consider again expression \eqref{eq:limsup_Fhp} with $t=t_{p,\infty}\in\mathcal{S}^{p}$. Using weak lower semicontinuity of $F_{h,p}$ it follows that
    \begin{align}
        \limsup_{q\to\infty}F_{h,p}(t_{p,q})\leq F_{h,p}(t_{p,\infty})\leq \liminf_{q\to\infty}F_{h,p}(t_{p,q})\;,
    \end{align}
    i.e. $F_{h,p}(t_{p,q})\to F_{h,p}(t_{p,\infty})$ as $q\to\infty$. Hence, the sequence $(\|X_{q}\|_{H^{p-1}(\Omega,h)})_{q\in\N}$ converges to  $\|X\|_{H^{p-1}(\Omega,h)}$, with $X_{q}:=u-\nabla^{g}t_{p,q}$ and $X_{\infty}:=u-\nabla^{g}t_{p,\infty}$, and by the previous proposition $X_{q}$ also converges weakly in $H^{p-1}$ to $X_{\infty}$.
    Since $H^{p-1}(\Omega)$ is a scalar product space, $X_{q}$ converges also strongly in $H^{p-1}$ to $X_{\infty}$ as follows from a direct inspection of the following expression
    \begin{align*}
        \|X_{q}-X_{\infty}\|_{H^{p-1}(\Omega)}^{2}=\|X_{q}\|_{H^{p-1}(\Omega)}^{2}+ \| X_{\infty}\|^{2}_{H^{p-1}(\Omega)}-2(X_{q},X_{\infty})_{H^{p-1}(\Omega)}\;,
    \end{align*}
    where the right hand side converges to $0$ as $q\to\infty$ by convergence of the norms and weak convergence in $H^{p-1}(\Omega)$. Strong $H^{p-1}(\Omega)$ convergence of the sequence of gradients together with strong $H^{1}(\Omega)$ convergence of the sequence of functions, allows us to upgrade the latter to strong $H^{p}(\Omega)$ convergence.\\
    For the second claim, let $\varphi\in C^{\infty}_{0}(\Omega^{\circ})$ and denote its support by $K$. By the proof of Proposition \ref{prop:smooth_minimizer} it is not necessary that they satisfy the zero-mean condition because the functional only depends on weak derivatives of $t\in\mathcal{S}^{p}$. Then, there exists a sufficiently small $s'\in [0,1]$ such that $t_{p,\infty}+s\varphi\in \mathcal{S}^{p}$ for all $|s|\in[0,s']$ (obvious in $\Omega\setminus K$). Indeed, condition \eqref{eq:infimum_tpinfty} implies existence of a $\delta>0$ such that $|\nabla^{g}t_{p,\infty}|_{g}(x)\geq 1+\delta$ for all $x\in\Omega$ and so in $K$ it holds that
    \begin{align*}
    |\nabla^{g}t_{p,\infty}+s\nabla^{g}\varphi|_{g}^{2}&=|\nabla^{g}t_{p,\infty}|_{g}^{2}-2sg(\nabla^{g}t_{p,\infty},\nabla^{g}\varphi)-s^{2}g(\nabla^{g}\varphi,\nabla^{g}\varphi)\\
    &\geq|\nabla^{g}t_{p,\infty}|_{g}^{2}-2|s| |g(\nabla^{g}t_{p,\infty},\nabla^{g}\varphi)|-s^{2}|g(\nabla^{g}\varphi,\nabla^{g}\varphi)|\\
        &\geq(1+\delta)^{2}-|s|C(2|\nabla^{g}t_{p,\infty}|_{h}|\nabla^{g}\varphi|_{h}+|s||\nabla^{g}\varphi|_{h}^{2})\\
        &\geq(1+\delta)^{2}-|s|C(C')^{2}(2+|s|)\;,
    \end{align*}
    where in the third line we simply used that since $g$ is smooth and the support of $\varphi$ is compact there exists a constant $C>0$ such that $|g(\nabla^{g}t_{p,\infty},\nabla^{g}\varphi)|\leq C|\nabla^{g}t_{p,\infty}|_{h}|\nabla^{g}\varphi|_{h}$ in $K$ and in the fourth line we exploited again compactness of $K$ and continuity of $\nabla^{g}t_{p,\infty}$ and $\nabla^{g}\varphi$ to upper bound their norms by a constant $C'$. Hence, one can now choose an $s'\in[0,1]$ sufficiently small such that
    \begin{align}\label{eq:s_estimate}
        s'C(C')^{2}(2+s')< (1+\delta)^{2}-1\;,
    \end{align}
    and then $|\nabla^{g}t_{p,\infty}+s'\nabla^{g}\varphi|_{g}^{2}>1$ in $\Omega$, so $t_{p,\infty}+s'\varphi\in \mathcal{S}^{p}$. Note that the function $f(s'):=s'(2+s')$ is strictly increasing for $s'\geq0$ ($f'(s')=2+2s'>0$) so \eqref{eq:s_estimate} is satisfied by any $|s|\in[0,s']$ and thus $t_{p,\infty}+s\varphi\in \mathcal{S}^{p}$ for all $s\in[-s',s']$.
    Then, the Gateaux derivative of $F_{h,p}$ at $t_{p,\infty}$ in the direction of $\varphi$ vanishes and, using the proof of Proposition \ref{prop:smooth_minimizer}, the following weak Euler-Lagrange equation is satisfied on $\Omega^{\circ}$
\begin{align}
    \sum_{i=0}^{p-1}\nabla^{g\ast}(\nabla^{h\ast})^{i}(\nabla^{h})^{i}(\nabla^{g} t_{p,\infty}-u)=0\;,
\end{align}
    which is again uniformly elliptic, yielding smoothness of $t_{p,\infty}$ in the interior of $\Omega$.
\end{proof}

Finally, the following corollary indicates how the minimizer of the misalignment $F_{h,p}$ with respect to the vector field $u$ and $t_{p,\infty}$ are related. As before, for $t\in\Tspace$ consider the following functional:
\begin{align}
        E(t):=\underset{x\in\Omega}{\text{\rm{ess inf}}}\;|\nabla^{g}t|_{g}(x)\;.
\end{align}

\begin{Corollary}\label{cor:pointwise_bound}
    Let $p\geq2$ and $\tilde{t}_{p}$ denote the minimizer of $F_{h,p}$ on $\Tspace$:
    \begin{itemize}
        \item[(i)] If $E(\tilde{t}_{p})\geq 1$, then $\tilde{t}_{p}=t_{p,\infty}$.
        \item[(ii)] If $E(\tilde{t}_{p})< 1$, then $E(t_{p,\infty})= 1$.
    \end{itemize}
    In particular, if $u$ is of gradient form, then 
    \begin{align*}
        E(t_{p,\infty})\geq \underset{x\in \Omega}{\text{\rm{inf}}}\;|u|_{g}(x)\;.
    \end{align*}
\end{Corollary}
\begin{proof}
    For the first claim note that if $\tilde{t}_{p}$ is almost everywhere steep, then $\tilde{t}_{p}\in \mathcal{S}^{p}$ and since it minimizes $F_{h,p}$ over $\Tspace$ it also minimizes the functional over $\mathcal{S}^{p}$. Uniqueness of the minimizer and Proposition \ref{prop:q_infinity} then imply that $\tilde{t}_{p}=t_{p,\infty}$.
    \par Assume that $E(\tilde{t}_{p})<1$ and $E(t_{p,\infty})>1$. Since $\Tspace$ is convex, any convex combination of $\tilde{t}_{p}$ and $t_{p,\infty}$ belongs again to $\Tspace$. Applying the reversed triangle inequality to $t_{\lambda}:=\lambda \tilde{t}_{p}+(1-\lambda)t_{p,\infty}$, with $\lambda\in[0,1]$, leads to
    \begin{align*}
        E(t_{\lambda})\geq \lambda E(\tilde{t}_{p})+(1-\lambda)E(t_{p,\infty})\to E(t_{p,\infty})>1 \quad \textrm{as}\quad\lambda\to0^{+}\;.
    \end{align*}
    Hence, there exists a $\lambda'\in[0,1]$ such that $E(t_{\lambda})\geq 1$ for all $\lambda\in[0,\lambda']$, which implies that $t_{\lambda}\in \mathcal{S}^{p}$ for $\lambda\in [0,\lambda']$. On the other hand, using strict convexity of $F_{h,p}$ we have that
    \begin{align*}
        F_{h,p}(t_{\lambda})<\lambda F_{h,p}(\tilde{t}_{p})+(1-\lambda)F_{h,p}(t_{p,\infty})< F_{h,p}(t_{p,\infty})\;,
    \end{align*}
    where we used that $\tilde{t}_{p}$ is the minimizer of $F_{h,p}$ over $\Tspace$ and $\tilde{t}_{p}\neq t_{p,\infty}$. However, for $\lambda\in[0,\lambda']$, $t_{\lambda}\in \mathcal{S}^{p}$ so the above expression contradicts that $t_{p,\infty}$ is the minimizer of $F_{h,p}$ over $\mathcal{S}^{p}$. Therefore, $E(t_{p,\infty})=1$.

    Finally, if $u=\nabla^{g}f$ and $f$ satisfies the zero-mean condition (if not, one can subtract a constant such that it does satisfy it), then $f$ is the unique minimizer of $F_{h,p}$ on $\Tspace$ by strict convexity of the functional. Then, part (i) and (ii) imply that 
 \begin{align*}
     E(t_{p,\infty})\geq \inf_{x\in\Omega}|u|_{g}(x)\;.
 \end{align*}
\end{proof}
In the following remark we summarize the different results proven in this paper concerning steepness of the gradient of minimizers of the misalignment functional and compare it to the steepness of $u$. 
\begin{Remark}\label{rem:bounded_gradient}\rm{
 Even if $u$ is of gradient form, with $u=\nabla^{g}f$ and $f\in\Tspace$, in general the minimizer $t_{p,q}$ of $\Fpq$ is different from $f$ for $\alpha>0$. However, $\nabla^{g}t_{p,q}$ presents an improved average steepness with respect to $u$ in the following sense
 \begin{align*}
     \alpha F_{g,q}&(t_{p,q})\leq \Fpq(t_{p,q})\leq F_{h,p}(f)+\alpha F_{g,q}(f)=\alpha F_{g,q}(f)\nonumber\\
     &\implies \int_{\Omega}\frac{1}{|\nabla^{g}t_{p,q}|_{g}^{2q}}d\mu_{h}\leq\int_{\Omega}\frac{1}{|u|_{g}^{2q}}d\mu_{h}\;.
 \end{align*}
 In general, one cannot improve the previous expression to a pointwise comparison between the norms of $\nabla^{g}t_{p,q}$ and $u$. However a global estimate is available when considering the limit $q\to\infty$ as the previous corollary showed.
 \par What happens if $u$ is not of gradient form? Also in this scenario one can prove as steepness bound for the gradient of $t_{p,\infty}$ in terms of $u$ by modifying slightly the null gradient penalization functional. In particular, assume that 
 \begin{align*}
     F_{g,q}(t):=\int_{\Omega}\frac{m^{2q}}{|\nabla^{g}t|_{g}^{2q}}d\mu_{h}\;,
 \end{align*}
 where $m:=\inf_{x\in\Omega}|u|_{g}$. Let $t_{p,q}$ be again the minimizer of $\Fpq$ for fixed $p\geq2, q\in\N$ and $t_{p,\infty}$ the limit of the sequence $(t_{p,q})_{q\in\N}$. Then, for any past-directed smooth timelike vector field $u$ (i.e. even if it is not of gradient form) it holds that
 \begin{align}\label{eq:tpinfty_steepest}
     \underset{x\in\Omega}{\text{\rm{ess inf}}}\;\big(|\nabla^{g}t_{p,\infty}|_{g}(x)\big)\geq \inf_{x\in\Omega}(|u|_{g}(x))\;.
 \end{align}
 Indeed, also in this case the sequence $(\Fpq(t_{p,q}))_{q\in\N}$ is 
 bounded: if $m\leq1$ this is clear and if $m>1$, one can choose a temporal function $\tau$ satisfying that $|\nabla^{g}\tau|_{g}\geq m$ and then $\Fpq(t_{p,q})\leq \Fpq(\tau)\leq C$ (where $C$ is a $q$-independent constant). Expression \eqref{eq:tpinfty_steepest} now follows from \eqref{eq:limsup_2} by assuming that $E(t_{p,\infty})<m$ and choosing an $\varepsilon>0$ such that $E(t_{p,\infty})+2\varepsilon<m$.
}\end{Remark}

Finally, since $q$ quantifies the strength of the null gradient penalization, one might expect that $E(t_{p,q})$ increases monotonically in $q$, which would yield a further improvement of Proposition \ref{prop:q_infinity}. Nevertheless, the following example shows that, in general, such a statement does not hold.
\begin{Example}\label{ex:cylinder}\rm{
Consider the two dimensional cylinder $M:=\R\times \mathbb{S}$ with metric
\begin{align*}
    g=-ds^{2}+d\theta^{2}\;.
\end{align*}
Moreover, consider the vector field $u:=-a \partial_{s}$ with $a>0$ constant (so $h=\delta$ is the flat metric) and the subset $\Omega:=I\times \mathbb{S}\subset M$ with $I$ a compact interval. Since rotations leave $u$ and $\Omega$ invariant, Proposition \ref{prop:symmetric} implies that the minimizer of $\Fpq$ (with $p\geq 2$) is $\theta$-independent: given a rotation $\Phi_{\varphi} : M\to M,(s,\theta)\mapsto (s,\theta+\varphi \;(\textrm{mod }2\pi))$, $t_{p,q}(s,\theta)=t_{p,q}(\Phi_{\varphi}(s,\theta))$ for all $\varphi\in \R$. Thus, the minimizer is of the following form:
\begin{align}\label{eq:tx}
    t_{x}(s,\theta):=x(s)+C\;,
\end{align}
with $x\in H^{p}(I)$, $x'(s)>0$ ($\nabla^{g}t_{x}$ is past-directed timelike) and where $C$ is a constant which guarantees that $t_{x}$ satisfies the zero-mean condition. Moreover,
\begin{align}\label{eq:example_functional_ansatz}
    \Fpq(t_{x})&=\int_{\Omega}\Big((x'(s)-a)^{2}+ \alpha (x'(s))^{-2q}+\sum_{i=2}^{p}|x^{(i)}(s)|^{2}\Big)d\mu_{h}\nonumber\\
    &=:\int_{\Omega}\Big(f(x'(s))+\sum_{i=2}^{p}|x^{(i)}(s)|^{2}\Big)d\mu_{h}\;.
\end{align}
The minimizer of $f$ satisfies that
\begin{align}\label{eq:fprime}
    &f'(x')=2(x'-a)-2\alpha q(x')^{-2q-1}\overset{!}{=}0\\
    &f''(x')=2+2q\alpha(2q+1)(x')^{-2q-2}>0\nonumber\;.
\end{align}
Note that $f'(x')\to-\infty$ as $x'\to0^{+}$ (recall that $x'>0$) and $f'(x')\to\infty$ as $x'\to\infty$, so by the intermediate value theorem and that $f'$ is strictly increasing it follows that the equation $f'(x')=0$ has a unique solution. We denote by $c_{\alpha,q}$ the value of $x'(s)$ which solves \eqref{eq:fprime}. Since it is independent of $s$, the minimizer of $f$ is $x'(s)=c_{\alpha,q}$ for all $s$ and the corresponding temporal function is $t_{c_{\alpha,q}}(s,\theta)=c_{\alpha,q}s+C$. Now note that since $h$ is the flat metric, $\nabla^{h}(u-\nabla^{g}t_{c_{\alpha,q}})=\nabla^{h}((c_{\alpha,q}-a)\partial_{s})=0$, so
\begin{align*}
    \Fpq(t_{c_{\alpha,q}})=\mu_{h}(\Omega)f(c_{\alpha,q})\;.
\end{align*}
Since $t_{c_{\alpha,q}}$ minimizes the first term in \eqref{eq:example_functional_ansatz} and the second one is non-negative, for any other $t_{x}$ of the form \eqref{eq:tx} we have that $\Fpq(t_{x})\geq \Fpq(t_{c_{\alpha,q}})$. Since the minimizer $t_{p,q}$ of $\Fpq$ is rotationally invariant it follows that for any $t\in\Tspace$, $\Fpq(t)\geq \Fpq(t_{c_{\alpha,q}})$ and so $t_{p,q}=t_{c_{\alpha,q}}$ (by uniqueness of minimizers) with $c_{\alpha,q}$ the solution to
\begin{align*}
    &f'(c_{\alpha,q})=2(c_{\alpha,q}-a)-2\alpha qc_{\alpha,q}^{-2q-1}=0\;,
\end{align*}
Moreover, by Proposition \ref{prop:q_infinity}, the sequence $(t_{p,q})_{q\in\N}$ converges strongly in $H^{1}(\Omega)$ to $t_{p,\infty}$. Since each $t_{p,q}=t_{c_{\alpha,q}}$ is affine-linear, also $t_{p,\infty}$ is affine-linear (by $L^{2}$-convergence of the derivatives). Let $c_{\alpha,\infty}>0$ be the constant such that $t_{p,\infty}=c_{\alpha,\infty}s+C=t_{c_{\alpha,\infty}}$.

Let $a<1$. The minimizer of $F_{h,p}$ is $\tilde{t}_{p}(s,\theta)=as$ (up to a constant) and $E(\tilde{t}_{p})=a<1$, so Corollary \ref{cor:pointwise_bound} implies that $E(t_{p,\infty})=c_{\alpha,\infty}=1$. Moreover, $f'(1)=2-2a-2\alpha q$, so whenever $\alpha q>1-a$ we have that $f'(1)<0$. Since $f$ is strictly convex, $f'$ is strictly increasing which together with $f'(1)<0$ for $q>\frac{1-a}{\alpha}$ and $f'(c_{\alpha,q})=0$ implies that $c_{\alpha,q}>1=c_{\alpha,\infty}$. Hence, for $q> \frac{1-a}{\alpha}$ it holds that
\begin{align*}
    \inf_{x\in\Omega}\big(|\nabla^{g}t_{c_{\alpha,q}}|_{g}(x)\big)=c_{\alpha,q}>c_{\alpha,\infty}=\inf_{x\in\Omega}\big(|\nabla^{g}t_{c_{\alpha,\infty}}|_{g}(x)\big)\;.
\end{align*}
and, since $c_{\alpha,q}\to c_{\alpha,\infty}$, the infimum of $|\nabla^{g}t_{c_{\alpha,q}}|_{g}$ cannot increase monotonically with $q$. If $a\geq1$, then $E(t_{c_{\alpha,\infty}})=c_{\alpha,\infty}=a$ (by Corollary \ref{cor:pointwise_bound}, (i)) and $f'(a)<0$ for all $q\in\N$. So in this case we even have that $E(t_{c_{\alpha,q}})=c_{\alpha,q}>a=c_{\alpha,\infty}=E(t_{c_{\alpha,\infty}})$ for all $q\in\N$.
}\end{Example}

\subsection{Stability of the minimizer and spacetime convergence}\label{subsec:stability}\hfill
\vspace{0.2cm}

In order to prove stability of the alignment time function under perturbations in the considered vector field an important obstacle appears: namely, the Riemannian metric $h$ and thus the set $\Tspace$ and the Sobolev norm depend on the vector field $u$. For this reason, in the following proposition we assume that $h$ is a fixed background Riemannian metric (which, as discussed in the previous section presents the drawback that the associated alignment time function depends not only on $u$ and $g$ but also, in addition, on the chosen Riemannian metric $h$). Then, the next proposition establishes Lipschitz stability of the alignment time function in the $H^{p}$ topology under $H^{p-1}$ perturbations in the vector field $u$ if $p$ and $q$ are large enough.

\begin{Prp}\label{prop:stability}
    Let $\gamma\in(0,1)$, $p> n/2+1+\gamma$, $q\geq n/\gamma$, $\alpha>0$. Consider two smooth past-directed timelike vector fields $u$ and $\tilde{u}$ and denote by $t_{\alpha}$ and $\tilde{t}_{\alpha}$ the associated minimizer of $\Fpq$ and $\Tilde{F}^{\alpha}_{p,q}$. Then, there exists a constant $C>0$ (which may depend on $g, h, \Omega$ and $p$ but not on $q$, $\alpha$, $u$ or $\tilde{u}$) such that
    \begin{align}
        &\|t_{\alpha}-\tilde{t}_{\alpha}\|_{H^{p}(\Omega)}\leq C \|u-\tilde{u}\|_{H^{p-1}(\Omega)}\;.
    \end{align}
\end{Prp}
\begin{proof}
    For $s\in(0,1)$, $t_{\alpha}+s(\tilde{t}_{\alpha}-t_{\alpha})\in\Tspace$ and $\tilde{t}_{\alpha}+s(t_{\alpha}-\tilde{t}_{\alpha})\in\Tspace$ by convexity of $\Tspace$. Since $p> n/2+1+\gamma$ and $q\geq n/\gamma$, the gradient of the minimizers is everywhere timelike and the functional $F_{g,q}$ is Gateaux differentiable. In particular, convexity of $F_{g,q}$ implies that
    \begin{align*}
    &F_{g,q}(t_{\alpha})\geq F_{g,q}(\tilde{t}_{\alpha})+DF_{g,q}(\tilde{t}_{\alpha})[t_{\alpha}-\tilde{t}_{\alpha}]\implies DF_{g,q}(\tilde{t}_{\alpha})[t_{\alpha}-\tilde{t}_{\alpha}] + DF_{g,q}(t_{\alpha})[\tilde{t}_{\alpha}-t_{\alpha}] \leq0
    \end{align*}
    Let $\tilde{F}_{h,p}(t):=\|\tilde{u}-\nabla^{g}t\|^{2}_{H^{p-1}(\Omega)}$ and recall that $F_{h,p}(t):=\|u-\nabla^{g}t\|^{2}_{H^{p-1}(\Omega)}$ . Then, since $t_{\alpha}$ and $\tilde{t}_{\alpha}$ are minimizers of $\Fpq$ and $\Tilde{F}^{\alpha}_{p,q}$, the above inequality implies that
    \begin{align*}
        0&\leq D\Fpq(t_{\alpha})[\tilde{t}_{\alpha}-t_{\alpha}]+D\tilde{F}^{\alpha}_{p,q}(\tilde{t}_{\alpha})[t_{\alpha}-\tilde{t}_{\alpha}] \leq DF_{h,p}(t_{\alpha})[\tilde{t}_{\alpha}-t_{\alpha}] +
        D\tilde{F}_{h,p}(\tilde{t}_{\alpha})[t_{\alpha}-\tilde{t}_{\alpha}] 
        \\
        &=2(\langle \nabla^{g}t_{\alpha}-u,\nabla^{g}(\tilde{t}_{\alpha}-t_{\alpha})\rangle_{H^{p-1}(\Omega)}+\langle \nabla^{g}\tilde{t}_{\alpha}-\tilde{u},\nabla^{g}(t_{\alpha}-\tilde{t}_{\alpha})\rangle_{H^{p-1}(\Omega)})\\
        &=2(\langle u-\tilde{u},\nabla^{g}(t_{\alpha}-\tilde{t}_{\alpha})\rangle_{H^{p-1}(\Omega)}-\|\nabla^{g}(t_{\alpha}-\tilde{t}_{\alpha})\|_{H^{p-1}(\Omega)}^{2})\;,
    \end{align*}
    where in third line we used expression \eqref{eq:Gateaux_diff_Fhp} expressed in terms of the Sobolev inner product. Combined with the Cauchy-Schwarz inequality results in
    \begin{align*}
        \|\nabla^{g}(t_{\alpha}-\tilde{t}_{\alpha})\|_{H^{p-1}(\Omega)}^{2}&\leq\langle u-\tilde{u},\nabla^{g}(t_{\alpha}-\tilde{t}_{\alpha})\rangle_{H^{p-1}(\Omega)}\\
        &\leq \|u-\tilde{u}\|_{H^{p-1}(\Omega)} \|\nabla^{g}(t_{\alpha}-\tilde{t}_{\alpha})\|_{H^{p-1}(\Omega)}\;.
    \end{align*}
    Applying now the Poincaré-Wirtinger inequality (both $t_{\alpha}$ and $\tilde{t}_{\alpha}$ satisfy the zero-mean condition) as in the coercivity proof in Lemma \ref{lem:props_functional} gives the desired inequality with a constant which might depend on $p$, $\Omega$, $g$ and $h$ but not on $q, \alpha, u$ or $\tilde{u}$.
\end{proof}
In the following remark we analyze some of the difficulties in proving an analogous statement to the one of the previous proposition if the metric $h$ is constructed from $g$ and $u$ using expression \eqref{eq:Riemannian_metric}.
\begin{Remark}\rm{
    Let $h_{u}$ and $h_{\tilde{u}}$ be the smooth Riemannian metrics associated to $u$ and $\tilde{u}$ respectively through expression \eqref{eq:Riemannian_metric}. Since $\Omega$ is compact, the associated Sobolev norms $\|\cdot\|_{H^{p}(\Omega,h_{u})}$ and $\|\cdot\|_{H^{p}(\Omega,h_{\tilde{u}})}$ are equivalent and, in particular, there exists a constant $C>0$ such that
    \begin{align*}
    \|h_{u}-h_{\tilde{u}}\|_{H^{p-1}(\Omega,h_{0})}=\|2(|u|_{g}^{-2}u^{\flat}\otimes u^{\flat}-|\tilde{u}|_{g}^{-2}\tilde{u}^{\flat}\otimes \tilde{u}^{\flat})\|_{H^{p-1}(\Omega,h_{0})}\leq C \|u-\tilde{u}\|_{H^{p-1}(\Omega,h_{0})}
    \end{align*}
    where $h_{0}$ is a fixed background Riemannian metric. However, the problem is that this constant $C$ depends on $u$ and $\tilde{u}$ and the same holds when comparing the norms $\|\cdot\|_{H^{p}(\Omega,h_{u})}$ and $\|\cdot\|_{H^{p}(\Omega,h_{\tilde{u}})}$. Hence, it is not clear whether one can prove in this case a Lipschitz stability result analogous to the one in the previous proposition with the constant $C$ independent of $u$ and $\tilde{u}$.

}\end{Remark}

For the rest of the paper we consider again that the Riemannian metric $h$ is constructed from the pair $(g,u)$ using \eqref{eq:Riemannian_metric}. In the remainder of this subsection we show that, given a convergent sequence of Lorentzian metrics and vector fields on $M$, the associated sequence of alignment time functions converges to the desired minimizer of the misalignment functional. In particular, we consider Sobolev spaces defined through different Riemannian metrics: given a Riemannian metric $h_{k}$ we denote the associated Sobolev space by $H^{m}(\Omega,h_{k})$ and the associated space of square integrable functions by $L^{2}(\Omega,h_{k})$. Recall that by \eqref{eq:Riemannian_metric} a sequence of Lorentzian metrics $(g_{k})_{k\in \N}$ and vector fields $(u_{k})_{k\in \N}$  give rise to a sequence of Riemannian metrics $(h_{k})_{k\in\N}$. For each $k\in \N$ the associated misalignment functional (if well defined) is
    \begin{align}\label{eq:sequence_functionals}
        F^{\alpha,k}_{p,q}(t):=F_{h_{k},p}(t)+\alpha F_{g_{k},q}(t)=\|u_{k}-\nabla^{g_{k}}t\|_{H^{p-1}(\Omega,h_{k})}^{2}+\alpha\int_{\Omega}\frac{1}{|\nabla^{g_{k}}t|_{g_{k}}^{2q}}d\mu_{h_{k}}\;,
    \end{align}
where $t\in \mathcal{T}^{p,k}\subset H^{p}(\Omega,h_{k})$ with $\mathcal{T}^{p,k}$ defined in full analogy to $\Tspace$ and the zero-mean condition is to be satisfied with respect to the measure $\mu_{h_{k}}$. Note that in the rest of this section $C$ and $C'$ denote $k$-independent constants which may change from one line to another.

\begin{Lemma}\label{lem:Sobolev_comparison}
Consider a sequence of Riemannian metrics $(h_{k})_{k\in \N}$ which converge in $C^{p}$-norm to the metric $h$. For sufficiently large $k$ there exist constants $C,C'$ independent of $k$ such that for any function $f\in H^{p}(\Omega,h)$ it holds that
\begin{align}\label{eq:inequality}
        C'\|f\|_{H^{p}(\Omega,h_{k})}\leq\|f\|_{H^{p}(\Omega,h)}&\leq C\|f\|_{H^{p}(\Omega,h_{k})}\;.
\end{align}
An analogous estimate holds for any $H^{p-1}$-regular vector field.
\end{Lemma}
\begin{proof}
    In the first place, the claim holds directly for the $L^2$ norms by $C^{0}$ convergence of $(h_{k})_{k\in\N}$: for sufficiently large $k\in\N$ there exist constants $a,b>0$ (independent of $k$) such that $a h\leq h_{k}\leq bh$ (analogous bounds hold for the associated volume forms and the tensor norms $|\cdot|_{h}$ and $|\cdot|_{h_{k}}$). Hence, there exist $k$-independent constants $C,C'$ such that for sufficiently large $k$ and square integrable tensor field $T$
    \begin{align}\label{eq:L2comparison}
        C'\|T\|_{L^{2}(\Omega,h_{k})}\leq\| T\|_{L^{2}(\Omega,h)}\leq C\|T\|_{L^{2}(\Omega,h_{k})}    \;.
    \end{align}
    and, in particular, it holds for $T=(\nabla^{h_{k}})^{j} f$ with $j\in\{0,\ldots,p\}$. Moreover, as $h_{k}\to h$ strongly in $C^{p}$, the following bounded operator
    \begin{align}\label{eq:difference_derivatives}
        (\nabla^{h})^{j}-(\nabla^{h_{k}})^{j}: H^{p}(\Omega,h)\to H^{p-j}(\Omega,h)\;,
    \end{align}
    where $j\in\{1,\ldots,p\}$, is a differential operator whose coefficients converge to $0$ in the $C^{p-j}$ norm and thus in the operator norm. So for any $f\in H^{p}(\Omega,h)$ and $j\in\{1,\ldots,p\}$
    \begin{align}
        \|(\nabla^{h})^{j}f-(\nabla^{h_{k}})^{j}f\|_{L^{2}(\Omega,h)}\leq\|((\nabla^{h})^{j}-(\nabla^{h_{k}})^{j})f\|_{H^{p-j}(\Omega,h)}\leq \varepsilon_{k}\|f\|_{H^{p}(\Omega,h)}\;,
    \end{align}
    where $\varepsilon_{k}\to0$ as $k\to\infty$. By the triangle inequality and the second inequality in \eqref{eq:L2comparison},
    \begin{align*}
        \| (\nabla^{h})^{j}f\|_{L^{2}(\Omega,h)}&\leq \|(\nabla^{h_{k}})^{j}f\|_{L^{2}(\Omega,h)}+\|(\nabla^{h})^{j}f-(\nabla^{h_{k}})^{j}f\|_{L^{2}(\Omega,h)}\\
        &\leq C\| (\nabla^{h_{k}})^{j}f\|_{L^{2}(\Omega,h_{k})}+\varepsilon_{k}\|f\|_{H^{p}(\Omega,h)}\;.
    \end{align*}
    Squaring and summing over $j=1,\ldots, p$ yields:
    \begin{align}
        \|f\|_{H^{p}(\Omega,h)}\leq C\|f\|_{H^{p}(\Omega,h_{k})}+C'\varepsilon_{k}\|f\|_{H^{p}(\Omega,h)}\;.
    \end{align}
    Hence, for sufficiently large $k$ (in particular, when $\varepsilon_{k}C'<1$) there exists a $k$-independent constant $C$ such that for any $f\in H^{p}(\Omega,h)$
    \begin{align*}
        \|f\|_{H^{p}(\Omega,h)}\leq C\|f\|_{H^{p}(\Omega,h_{k})}\;.
    \end{align*}
    On the other hand, a similar argument exploiting the first inequality in \eqref{eq:L2comparison} gives
    \begin{align*}
        \| (\nabla^{h_{k}})^{j}f\|_{L^{2}(\Omega,h_{k})}&\leq C\| (\nabla^{h_{k}})^{j}f\|_{L^{2}(\Omega,h)}\leq C\|(\nabla^{h})^{j}f\|_{L^{2}(\Omega,h)}+\varepsilon_{k}\|f\|_{H^{p}(\Omega,h)}\;,
    \end{align*}
    which readily implies the first inequality in \eqref{eq:inequality}. Finally, the same proof holds if one considers a $H^{p-1}$ regular vector field instead of a function.
\end{proof}

\begin{Prp}\label{prop:spacetime_convergence}
    Let $\gamma\in(0,1)$, $r\in\N$, $p>n/2+r+\gamma$, $q\geq n/\gamma$ and $\alpha>0$. Consider a sequence of Lorentzian metrics $(g_{k})_{k\in \N}$ and vector fields $(u_{k})_{k\in \N}$ which satisfy that
    \begin{align}
        g_{k}\to g,\quad u_{k}\to u \quad\textrm{strongly in }C^{p}\;.
    \end{align}
    Then, the sequence of minimizers $(t_{k})_{k\in \N}$ of the associated misalignment functionals \eqref{eq:sequence_functionals} converges strongly in $C^{r,\gamma}$ and $H^{p}$ to $t_{\alpha}$, the minimizer of $\Fpq$.
\end{Prp}
\begin{proof}
    The proof is divided in three parts: first, boundedness of the sequence of minimizers $(t_{k})_{k\in\N}$ is established, then $C^{r,\gamma}$-convergence to the minimizer $t_{\alpha}$ is proven, and, finally, this is upgraded to $H^{p}$-convergence.
    
    $C^{p}$ convergence of $(g_{k})_{k\in \N}$ and $(u_{k})_{k\in \N}$ implies $C^{p}$ convergence of the sequence of Riemannian metrics $(h_{k})_{k\in\N}$. Let $\tau$ be a smooth temporal function in $(M,g)$, $t_{k}$ the unique minimizer of $F^{\alpha,k}_{p,q}$ on $\mathcal{T}^{p,k}$ and $t_{\alpha}$ the minimizer of $\Fpq$. Up to an additive reparametrization guaranteeing that $\tau$ and $t_{\alpha}$ satisfy the zero-mean condition with respect to the measure $\mu_{h_{k}}$ and with a slight abuse of notation we have that $\tau,t_{\alpha}\in \mathcal{T}^{p,k}$. 
    
    The sequence $(\nabla^{g_{k}}t_{k})_{k\in\N}$ can be bounded by a $k$-independent constant $C$:
    \begin{align}\label{eq:boundtk}
        \|\nabla^{g_{k}}t_{k}\|_{H^{p-1}(\Omega,h_{k})}
        \leq C (1+(F^{\alpha,k}_{p,q}(\tau))^{1/2})<C'\;.
    \end{align}
    The first inequality follows from compactness of $\Omega$ and $C^{p}$ convergence of $u_{k}$
\begin{align*}
    \|\nabla^{g_{k}}t_{k}\|_{H^{p-1}(\Omega,h_{k})}&\leq \|u_{k}\|_{H^{p-1}(\Omega,h_{k})}+\|u_{k}-\nabla^{g_{k}}t_{k}\|_{H^{p-1}(\Omega,h_{k})}\leq C(1+ (F^{\alpha,k}_{p,q}(t_{k}))^{1/2}).
   \end{align*}
together with minimality of $t_{k}$. It remains to prove the second inequality in \eqref{eq:boundtk}. Since $g_{k}\to g$ uniformly, also $\nabla^{g_{k}}\tau\to \nabla^{g}\tau$ uniformly and $|\nabla^{g_{k}}\tau|_{g_{k}}\to |\nabla^{g}\tau|_{g}$ pointwise in $\Omega$. As $|\nabla^{g}\tau|_{g}\geq\delta$ for some $\delta>0$, there exists a $K\in\N$ with $|\nabla^{g_{k}}\tau|_{g_{k}}\geq \delta/2$ for all $k\geq K$. Since for sufficiently large $k$, $h_{k}\leq bh$, with $b>0$, we have that
    \begin{align}\label{eq:bound_tau}
        &F_{g_{k},q}(\tau)=\int_{\Omega}\frac{1}{|\nabla^{g_{k}}\tau|_{g_{k}}^{2q}}d\mu_{h_{k}}\leq C\int_{\Omega} (\delta/2)^{-2q}d\mu_{h}\leq C'\;.
    \end{align}
Finally, by Lemma \ref{lem:Sobolev_comparison} (applied to $u_{k}-\nabla^{g_{k}}\tau$) and $C^{p}$ convergence of $(g_{k})_{k\in\N}$ and $(u_{k})_{k\in\N}$ also $F_{h_{k},p}(\tau)$ can be bounded by a $k$ independent constant, concluding the proof of \eqref{eq:boundtk}. Boundedness of $(\nabla^{g_{k}}t_{k})_{k\in\N}$ in $H^{p-1}(\Omega,h_{k})$ can be upgraded to boundedness in $H^{p}(\Omega,h)$ of the sequence $(t_{k})_{k\in\N}$:
    \begin{align*}
        \|t_{k}\|_{H^{p}(\Omega,h)}\leq C\|t_{k}\|_{H^{p}(\Omega,h_{k})}\leq C\|\nabla^{h_{k}}t_{k}\|_{H^{p-1}(\Omega,h_{k})}\leq C\|\nabla^{g_{k}}t_{k}\|_{H^{p-1}(\Omega,h_{k})}\leq  C'\;,
    \end{align*}
    where the first inequality is due to Lemma \ref{lem:Sobolev_comparison}, the second one to the Poincaré-Wirtinger inequality (applied on $L^{2}(\Omega,h_{k})$), the third one to the $C^{p}$ convergence of the Riemannian metrics and compactness of $\Omega$, and the final one is \eqref{eq:boundtk}.
    By Banach-Alaoglu there exists a weakly converging subsequence in $H^{p}(\Omega)$ which, by the Sobolev embedding theorem (after passing to a further subsequence), also converges in the $C^{r,\gamma}$-norm to $t_{\ast}$. Hence, uniform convergence of the gradients implies that $\nabla^{g}t_{\ast}$  is a.e. causal and satisfies the zero-mean condition (so $t_{\ast}\in \Tspace$). Indeed, as $d\mu_{h_{k}}=\rho_{k}d\mu_{h}$ with $\rho_{k}$ positive and continuous and converging uniformly to 1, $t_{k}\rho_{k}$ converges in $L^{1}$ to $t_{\ast}$ and thus $\int_{\Omega}t_{\ast}d\mu_{h}=\lim_{k\to\infty}\int_{\Omega}t_{k}\rho_{k}d\mu_{h}=0$.
    
    The following inequality plays an important role in order to show that $t_{\ast}$ minimizes the functional $\Fpq$:
    \begin{align}\label{eq:liminf_Ftk}
        F^{\alpha}_{p,q}(t_{\ast})\leq \liminf_{k\to\infty}F_{p,q}^{\alpha,k}(t_{k})\;.
    \end{align}
    We proceed to prove this expression. In the first place, uniform convergence of $|\nabla^{g_{k}}t_{k}|_{g_{k}}$ to $|\nabla^{g}t_{\ast}|_{g}$, implies convergence of the non-negative functions $f_{k}:=(|\nabla^{g_{k}}t_{k}|_{g_{k}})^{-2q}\rho_{k}$ to the extended real valued function $f:=(|\nabla^{g}t_{\ast}|_{g})^{-2q}$ (a priori, $f$ might diverge since the gradient of $t_{\ast}$ can be null or vanish). By Fatou's lemma:
    \begin{align*}
        F_{g,q}(t_{\ast})=\int_{\Omega}\frac{1}{|\nabla^{g}t_{\ast}|_{g}^{2q}}d\mu_{h}\leq \liminf_{k\to\infty} \int_{\Omega}f_{k}d\mu_{h}= \liminf_{k\to\infty}F_{g_{k},q}(t_{k})\;.
    \end{align*}
    In particular, the right hand side is bounded by $F_{p,q}^{\alpha,k}(\tau)$ (since $t_{k}$ is the minimizer) and thus the gradient of $t_{\ast}$ is a.e. timelike. On the other hand, since $X_{k}:=u_{k}-\nabla^{g_{k}}t_{k}$ converges $H^{p-1}$-weakly to $X:=u-\nabla^{g}t_{\ast}$, also $(\nabla^{h_{k}})^{j}X_{k}$ converges weakly in $L^{2}$ to $(\nabla^{h})^{j}X$ for all $j\leq p-1$. Weak lower semicontinuity of the $L^{2}$ norm yields
    \begin{align*}
        F_{h,p}(t_{\ast})&=\sum_{i=0}^{p-1}\int_{\Omega} |(\nabla^{h})^{i}X|_{h}^{2}d\mu_{h}\leq \sum_{i=0}^{p-1}\liminf_{k\to\infty} \int_{\Omega} |(\nabla^{h_{k}})^{i}X_{k}|^{2}_{h}d\mu_{h}\\
        &\leq\sum_{i=0}^{p-1}\liminf_{k\to\infty} \int_{\Omega} |(\nabla^{h_{k}})^{i}X_{k}|^{2}_{h_{k}}d\mu_{h_{k}}\leq \liminf_{k\to\infty}F_{h_{k},p}(t_{k})\;,
    \end{align*}
    where the second inequality exploits that for an arbitrary tensor $T$ and sufficiently large $k$, $|T|^{2}_{h}\leq (1+\varepsilon_{k})|T|^{2}_{h_{k}}\rho_{k}$ with $\varepsilon_{k}\to0$ as $k\to\infty$.
    This concludes the proof of expression \eqref{eq:liminf_Ftk} which, together with the minimality of $t_{k}$, implies that
    \begin{align}\label{eq:convergence_seq}
    F^{\alpha}_{p,q}(t_{\ast})\leq \liminf_{k\to\infty}F_{p,q}^{\alpha,k}(t_{k})\leq
        \limsup_{k\to\infty} F_{p,q}^{\alpha,k}(t_{k})\leq \limsup_{k\to\infty} F_{p,q}^{\alpha,k}(t_{\alpha})= F^{\alpha}_{p,q}(t_{\alpha})\;.
    \end{align}
    and the last equality follows from $F^{\alpha,k}_{p,q}(t_{\alpha})\to F^{\alpha}_{p,q}(t_{\alpha})$. Since $t_{\alpha}$ is the unique minimizer of $F^{\alpha}_{p,q}$, it follows that $t_{\ast}=t_{\alpha}$. Hence, any subsequence of the sequence $(t_{k})_{k\in\N}$ admits a further subsequence converging $H^{p}$-weakly and $C^{r,\gamma}$-strongly to the unique minimizer $t_{\alpha}$ of $\Fpq$, which implies convergence of the full sequence.
    
    It remains to show strong $H^{p}$-convergence of the sequence $(t_{k})_{k\in\N}$. By \eqref{eq:convergence_seq}, $F^{\alpha,k}_{p,q}(t_{k})\to F^{\alpha}_{p,q}(t_{\alpha})$ so, in particular, $F_{h_{k},p}(t_{k})$ converges to $F_{h,p}(t_{\alpha})$. Together with $C^{p}$ convergence of $h_{k}$, it implies that the bounded sequence $(\|X_{k}\|_{H^{p-1}(\Omega,h)})_{k\in\N}$ converges to  $\|X\|_{H^{p-1}(\Omega,h)}$. It follows that $X_{k}$ converges to $X$ strongly in $H^{p-1}$ using the scalar product $(\cdot,\cdot)_{H^{p-1}(\Omega,h)}$:
    \begin{align*}
        \|X_{k}-X\|_{H^{p-1}(\Omega,h)}^{2}=\|X_{k}\|_{H^{p-1}(\Omega,h)}^{2}+ \| X\|^{2}_{H^{p-1}(\Omega,h)}-2(X_{k},X)_{H^{p-1}(\Omega,h)}\;,
    \end{align*}
    where the last expression also converges to $\|X\|_{H^{p-1}(\Omega,h)}^{2}$ by weak $H^{p-1}$ convergence of $X_{k}$ to $X$. Combined with $H^{1}$ convergence of $(t_{k})_{k\in\N}$ yields strong $H^{p}$ convergence.
\end{proof}
Note that it is expected that it is possible to lower considerably the demanded regularity for the sequence of Lorentzian metrics in the previous proposition. In particular, based on the discussion in Remark \ref{rem:optimal_regularity}, strong $H^{m}$ convergence, with $m\geq p$ and $m>n/2+1$, might be sufficient to obtain strong Sobolev convergence of the sequence of minimizers. This will be explored further in subsequent work.

It is important to remark that several recent notions of spacetime convergence rely on the choice of a time function. For instance, in \cite{Sakovich_Sormani_convergence} they employ the null distance, which requires the choice of a time function (in particular, they consider the cosmological time function, see also \cite{Sormani_Vega,Sakovich_Sormani_null_distance}). Another interesting example is the notion of convergence of globally hyperbolic spacetimes introduced in \cite{Burgos_Flores_Sanchez}, whose tools and setup share several important similarities with the ones used in this paper: in particular, it relies on the choice of a suitable Cauchy temporal function, which they use to construct a Riemannian metric (analogous to \eqref{eq:Riemannian_metric}) and enables them to exploit results from Riemannian geometry. It would be interesting to explore potential applications of the alignment time function to these different notions of spacetime convergence. This is discussed in more detail in Section \ref{sec:discussion}.

\subsection{Symmetry preservation}\hfill
\vspace{0.2cm}

The final subsection delves into the symmetry properties of the alignment time function. In particular, it is proven that, if the spacetime $(M,g)$ and the vector field $u$ enjoy suitable symmetry features, then the minimizer of the misalignment functional inherits the same properties.
\begin{Prp}\label{prop:symmetric}
    Let $p\geq2,q\in\N,\alpha\geq0$ and consider a $g$-isometry $\Phi: M\to M$ which preserves the subset $\Omega$ and the vector field $u$, i.e.
    \begin{align}
        \;\Phi(\Omega)=\Omega, \quad d\Phi|_{x} (u|_{x})=u|_{\Phi(x)} \quad \textrm{with}\quad x\in\Omega\;.
    \end{align}
    Then, the minimizer $t_{\alpha}$ of the misalignment functional $\Fpq$ satisfies that $t_{\alpha}\circ\Phi=t_{\alpha}$ almost everywhere in $\Omega$.
\end{Prp}
\begin{proof}
    The strategy of the proof is the following: under the above assumptions, it will be shown that for all $t\in \Tspace$
    \begin{align}
        t\circ\Phi\in\Tspace\quad\textrm{and}\quad\Fpq(t\circ\Phi)=\Fpq(t)\;,
    \end{align}
    which together with uniqueness of the minimizer implies that $t_{\alpha}=t_{\alpha}\circ\Phi$.
    \\In the first place, note that
    \begin{align}\label{eq:gradient_change1}
        &\nabla^{g}(t\circ \Phi)=(d\Phi)^{-1}(\nabla^{g}t\circ \Phi)\;.
    \end{align}
    Indeed for arbitrary $x\in \Omega$ and $Y\in T_xM$ the chain rule yields that
    \begin{align*}
     d(t\circ \Phi)|_{x}(Y)&=dt|_{\Phi(x)}(d\Phi|_{x}(Y))=g_{\Phi(x)}(\nabla^{g}t|_{\Phi(x)},d\Phi|_{x}(Y))\\
     &=g_{x}((d\Phi|_{x})^{-1}(\nabla^{g}t|_{\Phi(x)}),Y)\;,
    \end{align*}
    where in the last step we simply used that $\Phi$ is a $g$-isometry. Expression \eqref{eq:gradient_change1} then follows from noting that $d(t\circ \Phi)|_{x}(Y)=g_{x}(\nabla^{g}(t\circ \Phi)|_{x},Y)$.
    
    Using \eqref{eq:gradient_change1} and that isometries preserve the causal character, for any $t\in \Tspace$, $\nabla^{g}(t\circ \Phi)=d\Phi^{-1}(\nabla^{g}t\circ\Phi)$ is again causal. Since $d\Phi(u)=u$, the linear isometry $d\Phi$ preserves also time orientation of vectors with the same orientation as $u$: for a causal vector $v$ with the same time orientation as $u$, the same holds for $d\Phi(v)$ since
    \begin{align*}
        g(u,d\Phi(v))=g(d\Phi(u),d\Phi(v))=g(u,v)<0\;.
    \end{align*}
    Hence, $\nabla^{g}(t\circ \Phi)$ is past-directed causal. Furthermore, since $\Phi^{\ast}g=g$ and $d\Phi(u)=u$, $\Phi$ is also an isometry with respect to the Riemannian metric $h$ and thus preserves the measure $\mu_{h}$, i.e. $\Phi_{\ast}\mu_{h}=\mu_{h}$. As $\Phi$ preserves the subset $\Omega$ (i.e. $\Phi(\Omega)=\Omega$), it follows that $t\circ \Phi$ also satisfies the zero-mean condition
    \begin{align*}
        \int_{\Omega}(t\circ \Phi )d\mu_{h}=\int_{\Phi(\Omega)}t \;d(\Phi_{\ast}\mu_{h})=\int_{\Omega}t\;d\mu_{h}=0\;.
    \end{align*}
    Hence, $t\circ\Phi\in \Tspace$ for any $t\in \Tspace$. 
    
    On the other hand, expression \eqref{eq:gradient_change1} combined with $(d\Phi)^{-1}(u\circ \Phi)=u$ and that $\Phi$ is an $h$-isometry implies that the Riemannian functional $F_{h,p}$ remains invariant,
    \begin{align*}
         |u-\nabla^{g}(t\circ\Phi)|_{h}(x)&=|(d\Phi)^{-1}((u-\nabla^{g}t)\circ \Phi)|_{h}(x)=|u-\nabla^{g}t|_{h}(\Phi(x))\\
         |(\nabla^{h})^{i}(u-\nabla^{g}(t\circ\Phi))|^{2}_{h}(x)&=|(\nabla^{h})^{i}(u-\nabla^{g}t)|^{2}_{h}(\Phi(x))\\
        F_{h,p}(t\circ \Phi)&=\sum_{j=0}^{p-1}\int_{\Omega} |(\nabla^{h})^{j}(u-\nabla^{g}t)|_{h}^{2}\circ \Phi \;d\mu_{h}=F_{h,p}(t)\;,
    \end{align*}
    where $x\in \Omega,i\in\N$ and we used that the connection $\nabla^{h}$ behaves naturally under the pullback $d\Phi^{-1}$(i.e. $\nabla^{h}(d\Phi^{-1}(Y\circ \Phi))=d\Phi^{-1}(\nabla^{h}Y\circ \Phi)$ for any vector field $Y$). Finally, also the penalty term $F_{g,q}$ remains invariant:
    \begin{align*}
        F_{g,q}(t\circ\Phi)&=\int_{\Omega}\frac{1}{|\nabla^{g}(t\circ\Phi)|_{g}^{2q}}d\mu_{h}=\int_{\Omega}\Big(\frac{1}{|\nabla^{g}t|_{g}^{2q}}\circ\Phi\Big)\; d\mu_{h}=F_{g,q}(t)\;.
    \end{align*}
    Hence, for any $t\in\Tspace$, we have that
    \begin{align*}
        \Fpq(t\circ \Phi)=\Fpq(t)    \;.
    \end{align*}
    Uniqueness of the minimizer $t_{\alpha}$ implies that $t_{\alpha}\circ\Phi=t_{\alpha}$.
\end{proof}
    
What type of symmetries satisfy the assumptions of Proposition \ref{prop:symmetric} and thus preserve the minimizer of the misalignment functional? It is clear that this does not apply to translation symmetries since then the set $\Omega$ is not preserved. The following corollary constrains even more the class of symmetries to which Proposition \ref{prop:symmetric} applies.
\begin{Corollary}\label{cor:symmetry_implications}
    Assume there exists a group $G\subset \textrm{Isom}(M)$ for which each isometry $\Phi\in G$ satisfies the conditions of Proposition \ref{prop:symmetric}. Then,
    \begin{itemize}[leftmargin=2em]
        \item[(i)] Let $X\in\Gamma(TM)$ be the infinitesimal generator of a one-parameter family of isometries $(\Phi)_{s\in\R}\subset G$. Then, in the weak sense
        \begin{align}
        X(t_{\alpha})=0\;.
        \end{align}
        \item[(ii)] $G$ cannot act transitively on a non-empty open subset of $\Omega$.
        \item[(iii)] If $x\in \Omega$ is a fixed point of $\Phi\in G$, then $d\Phi_{x}$ has at least one fixed past-directed timelike vector. In particular, if $p>n/2+1$ then for any $\alpha\geq0$ and $q\in\N$ the corresponding alignment time function $t_{\alpha}$ satisfies that
        \begin{align*}
            d\Phi|_{x}(\nabla^{g}t_{\alpha}(x))=\nabla^{g}t_{\alpha}(x)    \;.
        \end{align*}
    \end{itemize}
\end{Corollary}
\begin{proof}
    The first claim simply follows from the fact that for each $s\in\R$, the isometry $\Phi_{s}$ preserves the minimizer $t_{\alpha}$, i.e. $t_{\alpha}\circ \Phi_{s}=t_{\alpha}$ a.e.. Thus, in the weak sense
    \begin{align*}
        X(t_{\alpha})=\frac{d}{ds}(t_{\alpha}\circ\Phi_{s})\Big|_{s=0}=0\;.
    \end{align*}
    Secondly, assume that $G$ acts transitively on a non-empty open subset $U\subset \Omega$ (so $\mu_{h}(U)>0$). Let $p>n/2$, for which there exists a continuous and unique minimizer $t_{\alpha}\in \Tt$ of $\Fpq$ for any $\alpha>0$. Then, for any $x,y\in U$ there exists an isometry $\Phi\in G$ such that $\Phi(x)=y$ which implies that $t_{\alpha}$ is constant in $U$ since
    \begin{align*}
        t_{\alpha}(y)=t_{\alpha}(\Phi(x))=t_{\alpha}(x)\;.
    \end{align*}
    In particular, it follows that $\nabla^{g}t_{\alpha}=0$ a.e. in $U$ contradicting that $\nabla^{g}t_{\alpha}$ is a.e. timelike for any $\alpha>0$.
    \par Finally, let $x\in \Omega$ be a fixed point of $\Phi$. Then, $u_{x}$ is a fixed timelike vector of $d\Phi_{x}$. For $p> n/2+1$, the gradient of the minimizer $\nabla^{g}t_{\alpha}$ is continuous which together with $t_{\alpha}\circ \Phi=t_{\alpha}$ in $\Omega$ and expression \eqref{eq:gradient_change1} proves the claim
    \begin{align*}
        &\nabla^{g}t_{\alpha}(x)=\nabla^{g}(t_{\alpha}\circ\Phi)(x)=(d\Phi|_{x})^{-1}(\nabla^{g}t_{\alpha}(\Phi(x)))=(d\Phi|_{x})^{-1}(\nabla^{g}t_{\alpha}(x))\\
        &\implies d\Phi|_{x}(\nabla^{g}t_{\alpha}(x))=\nabla^{g}t_{\alpha}(x)\;,
    \end{align*}
    i.e. $\nabla^{g}t_{\alpha}(x)$ is a fixed point of $d\Phi_{x}$. If $\alpha>0$, $p>n/2+1+\gamma$, $q\geq n/\gamma$ and $u_{x}\neq(\nabla^{g}t_{\alpha})(x)$, then $d\Phi_{x}$ has even two distinct fixed past-directed timelike vectors.
\end{proof}
Note that, in particular, the previous corollary does not rule out rotations with a timelike axis which leave $u$ invariant, since then the axis contains all fixed points of $\Phi$ and $d\Phi$ has a fixed timelike vector on the axis. Moreover, that $d\Phi_{x}$ has a fixed past-directed timelike vector when $\Phi$ has the fixed point $x\in \Omega$ implies that $d\Phi_{x}\neq-\textrm{id}_{T_{x}M}$.

Proposition \ref{prop:symmetric} can also be used to explore conditions under which $t_{\alpha}$ is a Cauchy temporal function. Since the minimizer of the misalignment functional is a temporal function for any $\alpha>0$, $p> n/2+1+\gamma$ and $q\geq n/\gamma$, its level sets are spacelike and acausal hypersurfaces (so, in particular, they are intersected at most once by every future-directed timelike curve, i.e. they are achronal). However, in general, these level sets might not be Cauchy hypersurfaces in $\textrm{int}{(\Omega)}$. In the following corollary we show that, under specific symmetry conditions and boundary separation properties of $\partial \Omega$, the minimizer $t_{\alpha}|_{\Omega^{\circ}}$ is a Cauchy temporal function on $\Omega^{\circ}$.

\begin{Corollary}\label{cor:Cauchy}
    Let $\Omega_{+},\Omega_{-}\subset \partial\Omega$ be two disjoint subsets intersected exactly once by any maximal causal curve in $\Omega$, $\alpha>0$, $\gamma\in(0,1)$, $p> n/2+1+\gamma$ and $q\geq n/\gamma$. Moreover, assume there exists a group $G\subset \textrm{Isom}(M)$ for which each isometry $\Phi\in G$ satisfies the conditions of Proposition \ref{prop:symmetric}. If $G$ acts transitively on each of the sets $\Omega_{+}$ and $\Omega_{-}$, then $t_{\alpha}|_{\Omega^{\circ}}$ is a Cauchy temporal function on $(\Omega^{\circ},g|_{\Omega^{\circ}})$.
    $\Omega$.
\end{Corollary}
\begin{proof}
    As discussed in the proof of the previous corollary, since $G$ acts transitively on $\Omega_{+}$ and $\Omega_{-}$ the minimizer is constant on these subsets. In particular, we denote by $c_{-},c_{+}\in\R$ its value on $\Omega_{-}$ and $\Omega_{+}$ respectively. Let $\gamma:[0,1]\to\Omega$ be a maximal future-directed causal curve. By assumption, its endpoints are on $\Omega_{\pm}$ so, without loss of generality, and since $t_{\alpha}$ is a time function $(t_{\alpha}\circ \gamma)(0)=c_{-}<c_{+}=(t_{\alpha}\circ \gamma)(1)$. As $t_{\alpha}\circ \gamma$ is strictly increasing and continuous, it takes any value in $(c_{-},c_{+})$ exactly once. So each hypersurface $\{x\in \Omega:t_{\alpha}(x)=c\}\cap\Omega^{\circ}$ with $c\in(c_{-},c_{+})$ is intersected exactly once by maximal causal curves in $\Omega$. Finally, every point in $x\in \Omega$ is reached by a maximal causal curve in $\Omega$ (with endpoints on $\Omega_{\pm}$), so $t_{\alpha}(x)\in(c_{-},c_{+})$ and thus the interior level sets $\{x\in \Omega:t_{\alpha}(x)=c\}\cap \Omega^{\circ}$ with $c\in(c_{-},c_{+})$ cover $\Omega^{\circ}$.
\end{proof}

\begin{Remark}\label{rem:cauchy}\rm{
    The proof of Corollary \ref{cor:Cauchy} implies that if there exist two disjoint subsets $\Omega_{+},\Omega_{-}\subset \partial\Omega$ intersected exactly once by any maximal causal curve in $\Omega$ and the symmetry assumption is replaced by demanding that
    \begin{align}\label{eq:time_separation}
        c_{-}:=\sup_{x\in \Omega_{-}}t_{\alpha}(x)<\Inf_{x\in \Omega_{+}}t_{\alpha}(x)=:c_{+}\;,
    \end{align}
    then, all the intermediate level sets $\{x\in\Omega:t_{\alpha}(x)=c\}\cap\Omega^{\circ}$, with $c\in(c_{-},c_{+})$, are Cauchy hypersurfaces in $\Omega^{\circ}$ (but they do not necessarily cover the full set $\Omega^{\circ}$). Indeed, let $\gamma:[0,1]\to \Omega$ be a maximal future-directed causal curve with $\gamma(0)\in \Omega_{-}$ and $\gamma(1)\in \Omega_{+}$ and let $c\in(c_{-},c_{+})$ be arbitrary. Then,
    \begin{align*}
        t_{\alpha}(\gamma(0))\leq c_{-}<c<c_{+}\leq    t_{\alpha}(\gamma(1))\;.
    \end{align*}
    Since $t_{\alpha}\circ \gamma$ is a strictly increasing continuous function, for each $c\in(c_{-},c_{+})$ there exists a unique $s\in(0,1)$ such that $t_{\alpha}(\gamma(s))=c$. These level sets do not necessarily cover $\Omega^{\circ}$ since for $c>c_{+}$ or $c<c_{-}$ the above argument does not apply and, in general, $\Omega^{\circ}\cap \{x\in\Omega:t_{\alpha}(x)=c \}$ may be non-empty. The right figure of Figure \ref{fig:two-images} is an example of a compact subset $\Omega$ of Minkowski spacetime for which condition \eqref{eq:time_separation} is satisfied and clearly any maximal (in $\Omega$) causal curve intersects $t_{\alpha}^{-1}(c)$ for any $c\in (c_{-},c_{+})$. This does not happen in the figure on the left.

    For instance, condition \eqref{eq:time_separation} is satisfied if the compact set $\Omega$ is the time-slab spanned by an auxiliary background Cauchy time function $\tau$ (i.e. $\Omega:=\{x\in M: a\leq \tau(x)\leq b\}$ with $a,b\in\R$) in a spacetime with compact Cauchy hypersurfaces and the zero-mean condition in the definition of $\Tspace$ is replaced with the trace condition
    \begin{align*}
        t|_{\partial\Omega}=\tau|_{\partial\Omega}\,.
    \end{align*}

    However, as discussed in Remark \ref{rem:properties}, the trace condition has the important drawback that, even if $u$ is of gradient form, in general there does not exist a function  $f\in H^{p}(\Omega)$ satisfying the trace condition and with $u=\nabla^{g}f$.
}\end{Remark}

\begin{figure}[t]
    \centering

    \begin{subfigure}[b]{0.58\textwidth}
        \centering
        \begin{overpic}[width=\linewidth]{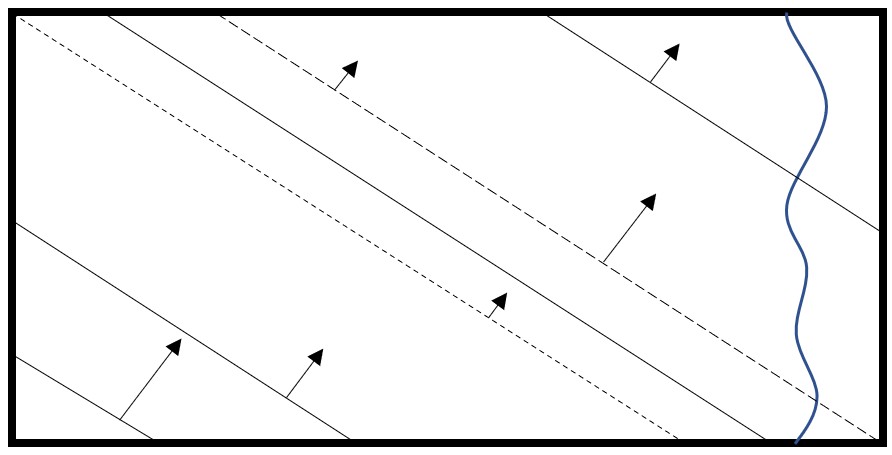}
            \put(48,53){\small $\Omega_{+}$}
            \put(48,4){\small $\Omega_{-}$}
            \put(23,20){\tiny $\displaystyle \{t_{\alpha}^{-1}(c_{+})\}$}
            \put(46,40){\tiny $\displaystyle \{t_{\alpha}^{-1}(c_{-})\}$}
            \put(38,8){\tiny $\displaystyle u$}
        \end{overpic}
    \end{subfigure}
    \hfill
    \begin{subfigure}[b]{0.41\textwidth}
        \centering
        \begin{overpic}[width=\linewidth]{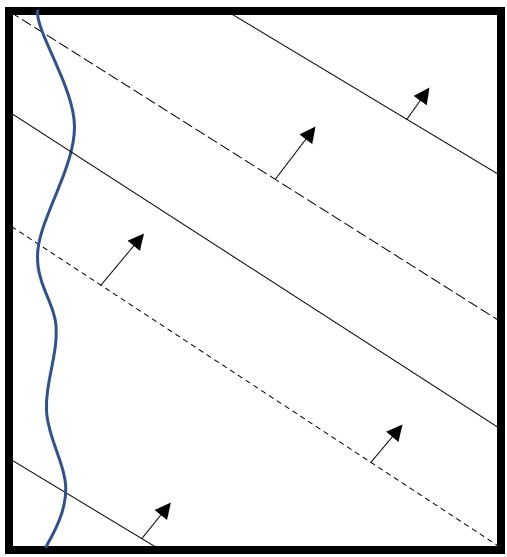}
            \put(45,103){\small $\Omega_{+}$}
            \put(45,4){\small $\Omega_{-}$}
            \put(30,85){\tiny$\displaystyle \{t_{\alpha}^{-1}(c_{+})\}$}
            \put(22,25){\tiny$\displaystyle \{t_{\alpha}^{-1}(c_{-})\}$}
            \put(34,9){\tiny$u$}
        \end{overpic}
    \end{subfigure}

    \caption{The example on the right satisfies condition \eqref{eq:time_separation}, the one on the left not.}
    \label{fig:two-images}
\end{figure}
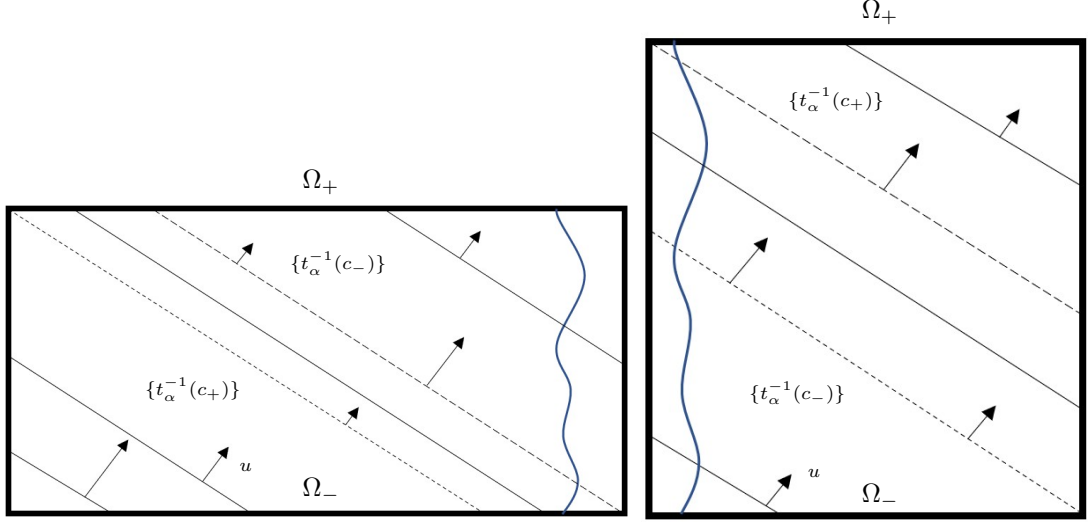

\section{Discussion and outlook}\label{sec:discussion}

This paper introduces a novel tool in Lorentzian Geometry: the alignment time function. More precisely, we have established existence and uniqueness of this temporal function in compact subsets of stably causal spacetimes and studied some of its properties. The alignment time function minimizes the misalignment with respect to a fixed timelike vector field $u$ and, in a suitable sense, presents an improved steepness. In particular, it provides a suitable temporal function to the physical models which assume the existence of a preferred timelike vector field.

Our work initiates a detailed analysis on temporal functions adapted, in a suitable sense, to fixed timelike vector fields. Nevertheless, several open questions, further extensions or generalizations of the obtained results and potential applications arise as natural next steps. The present work sets the stage for these future developments.

In the first place, it would be interesting to extend the presented results to the non-compact setting. In particular, asymptotically flat spacetimes seem well suited for the presented setup to work. Other classes of spacetimes which seem promising for this purpose are globally hyperbolic spacetimes whose Cauchy hypersurfaces have bounded geometry or spacetimes with compact Cauchy hypersurfaces. Of course, in the non-compact setting one would have to consider different spaces of functions (e.g. weighted Sobolev spaces) from those employed in the compact setting.

Secondly, recall that the results of Section \ref{sec:compact_setting} also hold for the Sobolev space $W^{p,r}(\Omega)$, with $r\in(1,\infty)$, instead of $H^{p}(\Omega)$. Consequently, a future research direction is to investigate whether on $W^{p,\infty}$ it is possible to prove existence of a temporal function which minimizes the following functional:
\begin{align*}
    F_{p,\infty}(t):=\max_{i\leq p-1}\|(\nabla^{h})^{i}(u-\nabla^{g}t)\|_{L^{\infty}(\Omega)}\;.
\end{align*}
Note that $F_{p,\infty}$ is not strictly convex (minimizers may be non-unique) and $L^{\infty}(\Omega)$ is not a reflexive Banach space. However, a minimizer of this functional (if it exists) presents at least two remarkable features: on the one hand, in stark contrast to $\Fpq$, the value of $F_{p,\infty}$ does not necessarily increase with the volume of $\Omega$. This turns $F_{p,\infty}$ into a natural candidate for the non-compact setting. On the other hand, set
\begin{align*}
    \delta:=\inf_{x\in\Omega}\big(d_{h}(u(x),\partial C_{x}^{-})\big)>0\;,
\end{align*}
where $C_{x}^{-}\subset T_{x}M$ denotes the set of past-directed timelike vectors. Then, assuming the minimizer exists, one can prove everywhere timelikeness of its gradient if $u$ is sufficiently close to a gradient vector field in the following sense. Namely, if there exists a function $f\in W^{p,\infty}(\Omega)$ (with $p\geq2$) such that  
\begin{align}\label{eq:Linftynorm}
    F_{p,\infty}(f)< \delta\;,
\end{align}
then $\nabla^{g}t_{\infty}$ is everywhere timelike (for $p=1$ one could still conclude a.e. timelikeness): since $p\geq 2$, $\nabla^{g}t_{\infty}$ is continuous (by the Sobolev embeddings) and condition \eqref{eq:Linftynorm} implies that $F_{p,\infty}(t_{\infty})\leq F_{p,\infty}(f)<\delta$, so $\|u-\nabla^{g}t_{\infty}\|_{L^{\infty}(\Omega)}<\delta$. Thus, $\nabla^{g}t_{\infty}(x)\in C_{x}^{-}$ for all $x\in\Omega$. However, in general, a null-gradient penalizing functional may still be needed to guarantee that the gradient of the minimizer is timelike.

In the third place, it is relevant to determine under which conditions (the ones of Corollary \ref{cor:Cauchy} are very restrictive), potentially on the vector field $u$, the alignment time function is a Cauchy temporal function. For example, one would expect that if $u$ is sufficiently close to the gradient of a Cauchy temporal function, then the alignment time function is Cauchy.

Furthermore, as discussed in Remark \ref{rem:optimal_regularity}, the misalignment functional is still well-defined if the metric $g$ and the vector field $u$ are not smooth. Hence, it would be interesting to determine for which lower regularity class for $g$ and $u$ the results of Section \ref{sec:compact_setting} still hold.

Finally, one further avenue for future research is to analyze possible applications of the alignment time function to the study of spacetime convergence. As mentioned at the end of Section \ref{subsec:stability}, our setup is closely related to the one used in \cite{Burgos_Flores_Sanchez}. They introduce both a notion of convergence of spacetimes with a privileged past-directed timelike vector field and one based on a choice of suitable Cauchy temporal functions. Recall that Proposition \ref{prop:spacetime_convergence} establishes stability of the alignment time function under convergence of sequences of Lorentzian metrics $g_{k}$ and vector fields $u_{k}$. Thus, it would be interesting to investigate whether the alignment time function is a suitable temporal function in their convergence framework. Ideally, our variational approach could provide a canonical procedure to single out a suitable temporal function from a vector field satisfying the required properties (for their convergence framework). Nevertheless, the main obstacles are that, in general, the alignment time function is not a Cauchy temporal function and that our analysis of the alignment time function has so far been restricted to compact subsets. Hence, in order to bridge the gap between our results and \cite{Burgos_Flores_Sanchez}, one would first need to address the question of which conditions ensure that the alignment time function is a Cauchy temporal function as well as the existence of the alignment time function in the non-compact setting.

\Thanks{{{\em{Acknowledgments:}}}} I am grateful to Felix Finster for many fruitful discussions and valuable suggestions. I would also like to thank Miguel Sánchez Caja and Micha\l{} Eckstein for stimulating exchanges. I thankfully acknowledge support by the Studienstiftung des deutschen Volkes.

\end{document}